\documentclass[a4paper,11pt,twoside,english]{article}
\usepackage{marginnote}
\usepackage[utf8x]{inputenc}
\usepackage[T1]{fontenc}
\usepackage{color}
\usepackage{lmodern}
\usepackage{graphicx}
\usepackage{amsmath,amssymb,amsthm}
\usepackage{mathrsfs}  
\usepackage{dsfont} 

\usepackage{caption}
\usepackage{subcaption}
\usepackage{enumitem}
\usepackage{setspace}
\usepackage{xspace} 
\usepackage{blkarray}
\usepackage[subnum]{cases}
\usepackage{stmaryrd}
\usepackage{bbm} 
\usepackage[margin=2cm]{geometry}

\usepackage[unicode=true]{hyperref}
\usepackage[all]{hypcap}

\usepackage{cleveref} 

\newcommand{\titel}{Spectrum occupies pseudospectrum for  random matrices with diagonal deformation and variance profile }

\hypersetup{
	pdftitle={\titel},
	pdfauthor={Johannes Alt, Torben Krüger}
}

\numberwithin{equation}{section}

\newcommand{\R}{\mathbb{R}}  
\ifdefined\C\renewcommand{\C}{\mathbb{C}}\else\newcommand{\C}{\mathbb{C}}\fi 
\renewcommand{\Im}{\mathrm{Im}\,} 
\newcommand{\N}{\mathbb{N}}  
\newcommand{\E}{\mathbb{E}}  
\renewcommand{\P}{\mathbb{P}}  
\newcommand{\eps}{\varepsilon} 
\newcommand*{\defeq}{\mathrel{\vcenter{\baselineskip0.5ex \lineskiplimit0pt\hbox{\scriptsize.}\hbox{\scriptsize.}}}=}

\DeclareMathOperator{\supp}{supp}

\DeclareMathOperator{\ran}{ran}

\DeclareMathOperator{\smin}{s_{\min}}

\newcommand{\DD}{\mathbb{D}}

\newtheoremstyle{test}
  {}
  {}
  {\itshape}
  {}
  {\bfseries}
  {.}
  { }
  {}

\theoremstyle{test}
\newtheorem{defi}{Definition}[section]

\newtheorem*{rem*}{Remark}   
\newtheorem*{ex*}{Example}   
\newtheorem*{pro*}{Proposition} 
\newtheorem*{def*}{Definition}
\newtheorem*{coro*}{Corollary}
\newtheorem*{thm*}{Theorem}

\theoremstyle{test}
    \newtheorem{theorem}[defi]{Theorem}
    \newtheorem{proposition}[defi]{Proposition}
    \newtheorem{corollary}[defi]{Corollary}
    \newtheorem{lemma}[defi]{Lemma}
    \newtheorem{definition}[defi]{Definition}
    
    \newtheorem{convention}[defi]{Convention}
    \newtheorem{remark}[defi]{Remark}

\newcommand{\bels}[2] {
        \begin{equation} \label{#1} \begin{split} 
                #2 
        \end{split} \end{equation}
        }

\newcommand{\bbm}{\mathbbm} 
\renewcommand{\cal}{\mathcal} 
\renewcommand{\frak}{\mathfrak} 
\newcommand{\ol}[1]{\overline{#1} \!\,} 
\newcommand{\wh}{\widehat}
\newcommand{\wt}{\widetilde}

\renewcommand{\P}{\mathbb{P}}

\newcommand{\ii}{\mathrm{i}} 
\newcommand{\dd}{\mathrm{d}}

\newcommand{\cbb}[1]{\biggl\{{#1}\biggr\}}

\newcommand{\abs}[1]{\lvert #1 \rvert}
\newcommand{\absb}[1]{\big\lvert #1 \big\rvert}
\newcommand{\absB}[1]{\Big\lvert #1 \Big\rvert}
\newcommand{\absbb}[1]{\bigg\lvert #1 \bigg\rvert}

\newcommand{\norm}[1]{\lVert #1 \rVert}

\newcommand{\avg}[1]{\langle #1 \rangle}

\newcommand{\db}[1]{\llbracket #1 \rrbracket}

\DeclareMathOperator{\diag}{diag}
\DeclareMathOperator{\tr}{Tr}
\DeclareMathOperator{\Tr}{Tr}

\DeclareMathOperator{\im}{Im}

\DeclareMathOperator{\dist} {dist}                
\DeclareMathOperator{\spec}{Spec}						

\newcommand{\1} {\mspace{1 mu}}
\newcommand{\2} {\mspace{2 mu}}

\newcommand{\qq}[1]{\llbracket #1 \rrbracket} 

\newcommand{\mtwo}[2]
{
\left(
\begin{array}{cc}
#1 
\\
#2
\end{array}
\right)
}

\makeatletter
 \newcommand{\linkdest}[1]{\Hy@raisedlink{\hypertarget{#1}{}}}
\makeatother

\makeatletter
\def\blfootnote{\xdef\@thefnmark{}\@footnotetext}
\makeatother

\begin{document}
\blfootnote{Date: \today}
\blfootnote{MSC2010 Subject Classifications: 60B20, 15B52.}

\title{\vspace{-0.8cm}{\textbf{\titel}}} 
\author{Johannes Alt$^{\text{a,} }$\thanks{Funding from the Deutsche Forschungsgemeinschaft (DFG, German Research Foundation) under Germany's Excellence Strategy - GZ 2047/1, project-id 390685813 is gratefully acknowledged. \newline \hspace*{0.5cm} Email: \href{mailto:johannes.alt@iam.uni-bonn.de}{johannes.alt@iam.uni-bonn.de}} 
 \hspace*{0.5cm} \and \hspace*{0.5cm} Torben Krüger$^\text{b,}$\thanks{
Financial support 
from VILLUM FONDEN Young Investigator Award  (Grant No. 29369) is gratefully acknowledged.  
\newline \hspace*{0.5cm} Email: \href{mailto:torben.krueger@fau.de}{torben.krueger@fau.de}}}
\date{}

\maketitle

\vspace*{-0.2cm} 

\noindent \hspace*{3.54cm} $^\text{a}$Institute for Applied Mathematics, University of Bonn \\ 
\noindent \hspace*{3.54cm} $^\text{b}$Department of Mathematics, FAU Erlangen-Nürnberg \\

\begin{abstract}
We consider 
 $n\times n$ non-Hermitian random matrices with independent entries and a variance profile, as well as an additive  deterministic diagonal deformation.
  We show that their empirical eigenvalue distribution converges to a limiting density as $n$ tends to infinity and that the support of this density  in the complex plane exactly coincides  with the $\eps$-pseudospectrum in the consecutive limits $n \to \infty$ and $\eps \to 0$. 
   The limiting spectral measure is identified as the Brown measure of a deformed operator-valued circular element with the help of \cite{AK_Brown}.
\end{abstract}


\section{Introduction} 
The celebrated circular law asserts that the empirical spectral distribution (ESD) of a random  matrix $X\in \C^{n \times n}$ with centered i.i.d.\ entries  $x_{ij}$ of variance $\E\1 \abs{x_{ij}}^2=\frac{1}{n}$ converges to the uniform distribution  on the complex unit disk as $n$ tends to infinity \cite{Girko1984,bai1997,tao2010} (see  \cite{BordenaveChafai2012} for a review). 
The convergence of the ESD to a non-random radially symmetric probability measure $\sigma$, supported on a disk, generalises to  the case when  the entries   of $X$  remain independent but admit differing distributions  with entry dependent variances $s_{ij}:=\E\1 \abs{x_{ij}}^2$ \cite{Cook2018,Altcirc} and even to  correlated entries with a decaying correlation structure \cite{AK_Corr_circ}. In both cases $\sigma$ is no longer uniform on its support in general.

In this work we consider a model in which a diagonal deterministic deformation $A = \diag(a_i)_{i=1}^n$ is added to a matrix $X$ with independent entries and variance profile $S=(s_{ij})_{i,j=1}^n$. The deformation $A$ breaks the radial symmetry of the limiting ESD $\sigma$. When $s_{ij} = \frac{1}{n} s(\frac{i}{n}, \frac{j}{n})$ and $a_i = a(\frac{i}{n})$ are discretisations of bounded profile functions $s:[0,1]^2 \to \R$ and $a:[0,1] \to \R$, respectively, the  measure $\sigma$ is realised as the Brown measure of an element $a + \frak{c} \in \cal{A}$ in a $L^{\infty}[0,1]$-valued noncommutative probability space  $(\cal{A}, L^{\infty}[0,1],E)$ within free probability theory.  Here $\cal{A}$ is a $W^*$-probability space $\cal{A}$ with a faithful tracial state, $L^{\infty}[0,1] \subset \cal{A}$ is a subalgebra and 
 $E: \cal{A} \to L^{\infty}[0,1]$ a conditional expectation. 
The Brown measure is a generalisation of the spectral measure to non-normal operators~\cite{Brown1986,HaagerupLarsen2000}.

When the variance profile function $s$ is constant, this model describes adding the deformation $A$ to a matrix with i.i.d. entries. On the level of random matrices the  ESD  was originally computed  in \cite{Khoruzhenko1996}. On the free probability level 
 $\frak{c}\in \cal{A}$ is a circular element that is $\ast$-free  from $a$ in this situation \cite{Sniady2002}. The corresponding analysis was carried out in \cite{BordenaveCaputoChafai2014,BordenaveCapitaine2016}. More recently, detailed information about the regularity of the Brown measure for these deformed circular elements, such as existence of a density \cite{BelinschiYinZhong2024} and analyticity \cite{Zhong2021,HoZhong2023},   have been obtained.  A jump discontinuity at the  edge of the support of $\sigma$ has been established in \cite{ErdosJi2023}, where it has also been shown that around the zeros of the density within its support, $\sigma$
grows at most quadratically with a matching  lower bound on at least  a two-sided cone with apex at the zero. 
A complete classification of these singularities is given in  \cite{AK_Brown}. There the edge singularities are characterised by the local shape of the boundary of $\supp \sigma$ and the internal singularity by the local growth of the density away from the two-sided cone. 

When the variance profile $s$ is non-constant the non-normal matrix $X+A$ belongs to the Kronecker matrix class discussed in \cite{AEKN_Kronecker}.  
Non-normal  random matrices and a detailed understanding of their spectra play an important role in many applications, ranging from the stability analysis of food webs \cite{Allesina:2015ux,may1972will,Hastings1992} and  
quantum chaotic scattering  \cite{FyodorovSommers1997}  to investigating the transition to chaos in neuronal networks  \cite{Sompolinsky1988,PhysRevLett.97.188104}.  
A persistent challenge in the analytic study of such matrices $X$ is their spectral instability, i.e.\ the fact that tiny changes in the matrix entries may lead to large deviations of the eigenvalues. To remedy this issue the $\eps$-pseudospectrum $\spec_\eps(X)$ is  introduced (see e.g.\ \cite{TrefethenEmbree+2005} for an overview), which is stable under perturbations,  monotonically increasing in $\eps>0$ and contains the spectrum, namely $\bigcap_{\eps>0} \spec_\eps(X)= \spec(X)$. Especially for high dimensional $X=X_n \in \C^{n \times n}$ the dependence of $\spec_\eps(X_n)$ on small values of $\eps$ may  very unstably depend on $n$ (see e.g.\ \cite[Section~11.6.3]{MingoSpeicherBook} for the example of a shift operator). In particular, the eigenvalues may accumulate in a much smaller area than the asymptotic pseudospectrum $\spec_0^\infty:=\lim_{\eps \downarrow 0}\lim_{n \to \infty}\spec_\eps(X_n)$. 
In the case of Toeplitz matrices $A$ with a very small added randomness $X$ for example, the spectrum concentrates on curves given by the image of the unit circle by the Toeplitz symbol inside $\spec_0^\infty$  \cite{GuionnetWoodZeitouni2014,REICHEL1992153,BASAK_PAQUETTE_ZEITOUNI_2019,SV2020}.

In contrast, 
our main result shows that for matrices with independent entries and diagonal deformation the set $\spec_0^\infty$ coincides with the support of the limiting spectral measure  $\sigma$, i.e.\ that the spectrum occupies the entire $\eps$-pseudospectrum in the consecutive limits $n \to \infty$ and $\eps\downarrow 0$ under some regularity assumptions on the profile functions. 
In particular, $\spec_0^\infty = \supp \sigma$ stably depends on the expectation profile $a$ and the  variance profile $s$. As was shown in \cite{AK_Brown}, the density of the Brown measure is strictly positive on an open domain $\mathbb{S}:= \{  \beta  < 0 \} \subset \C$ with boundary $\partial \mathbb{S}= \{  \beta  = 0 \}$, where $\beta : \C \to \R$ is a continuous function that is real analytic in a neighbourhood of $\partial \mathbb{S}$. 
In an independent work \cite{CampbellCipolloniErdosJi2024} that was posted on arXiv on the same day as the first version of our work, the authors show that the identity $\spec_0^\infty = \supp \sigma$  holds for the case when the entries of $X$ are i.i.d., i.e.\ when $s$ is constant, and the limits $n \to \infty$, $\eps \downarrow 0$ in the definition of  $\spec_0^\infty$ can be taken simultaneously with an optimal dependence of $\eps$ on $n$. Allowing this $n$-dependence of $\eps=\eps_n$ is a prerequisite for proving  universality of the local edge statistics in  \cite{CampbellCipolloniErdosJi2024}. 
Amending their argument as described in \cite[Remark 2.15]{CampbellCipolloniErdosJi2024}
 and using \cite[Proposition~5.16 (iv)]{AK_Brown} to check a technical assumption, 
 the identity $\lim_{n \to \infty}\spec_{\eps_n}(X_n) = \supp \sigma$ can also be shown from the results in \cite{CampbellCipolloniErdosJi2024} in the setting with non-constant variance profile $s$ whenever $s$ is strictly bounded away from zero.
The benefit of the approach presented in our current work, which does not aim to track any dependence of $\eps$ on $n$, is that it allows for zero blocks in the variance profile. Thus, our results apply to non-Hermitan random band matrices \cite{Jain2021_non_Hermitian_random_band,Jana2022_nonHermitian_random_band} with block structure as long as the band remains large in the $n \to \infty$ limit, 
i.e.\ when the number of non-zero entries is proportional to $n^2$.

\section{Main results} \label{sec:main_results}

In this section, we state our assumptions and the main results. 
 In the following, we take $n \in \N$ and write $\db{n}$ for the discrete interval $\db{n} = \{ 1, \ldots , n \}$.

\newcounter{assumptions} 

\begin{enumerate}[label=\textbf{A\arabic*}]
\setcounter{enumi}{\value{assumptions}} 
\item \label{assum:independent_centered} \stepcounter{assumptions} 
\emph{Independent, centered entries:} 
The entries of $X=(x_{ij})_{i,j \in \db{n}}$ are independent and centered, i.e.\ $\{ x_{ij} \colon i,j \in \db{n}\}$ is a family of independent  random variables and $\E\1 x_{ij}=0$. 
Moreover, all moments  of the entries of $\sqrt{n}X$ are finite, i.e.\  there is a sequence of positive constants $C_\nu$ such that 
\bels{bounded moments}{
\E\,\abs{x_{i j}}^\nu\,\le\, C_\nu \2n^{-\nu/2}\,,
}
for all $i,j\in \llbracket n\rrbracket$ and $\nu \in \N$.  
\item \label{assum:smallest_singular_value} \stepcounter{assumptions}
\emph{Smallest singular value: } 
Let $X \in \C^{n\times n}$ be a random matrix. 
Suppose that there are constants $r_0 \in (0,1/2]$, $K_0 \geq 1$, $\alpha >0$, $\beta >0$ and $C >0$ such that 
\[ \P \big( \smin (X + Z) \leq n^{-1/2 - \beta} \big) \leq C n^{-\alpha}  \] 
for any deterministic diagonal matrix $Z  = \diag(z_1, \ldots, z_n)$ satisfying $r_0 \leq \abs{z_i} \leq K_0$ for all $i \in \db{n}$. 
\end{enumerate} 
In the following we will always assume that the  deformation $A_n =(a^{(n)}_{ij})_{i,j=1}^n$ and the variance profile of $X_n=(x^{(n)}_{ij})_{i,j=1}^n$ are discretisations   of  limiting profile functions $a \colon [0,1] \to \C$  and $s \colon [0,1]^2 \to [0,\infty)$, i.e.\ that
\begin{equation} \label{eq:a_ij_and_s_ij_relation_a_and_s} 
 a_{ij}^{(n)} = a\bigg(\frac{i}{n} \bigg)\delta_{ij},  
\qquad \qquad 
\E \abs{x_{ij}^{(n)}}^2 = \frac{1}{n} s\bigg(\frac{i}{n}, \frac{j}{n} \bigg)  
\end{equation} 
for all $i$, $j \in \db{n}$. 
For the following assumptions on the profile functions, we suppose that there are $K \in \N$ and 
disjoint intervals $I_1$, \ldots, $I_K \subset [0,1]$ of positive length such that $I_1 \cup \ldots \cup I_K = [0,1]$. 
\begin{enumerate}[label=\textbf{A\arabic*}]
\setcounter{enumi}{\value{assumptions}} 
\item \label{assum:primitive_upper_lower} \stepcounter{assumptions}
 \emph{Block-primitivity of variance profile: } 
 The matrix $Z = (z_{lk})_{l,k \in \db{K}} \in \{0,1\}^{K\times K}$ with entries $z_{lk} :=  \bbm{1}(s|_{I_l \times I_k} \ne 0)$ is primitive, i.e.\ there is $L \in \N$ such that $(Z^L)_{lk} \geq 1$ for all $l$, $k \in \db{K}$, and that the diagonal blocks of $s$ are 
nonzero, i.e.\ $z_{kk} =1$ for all $k \in \db{K}$. 
Moreover, we suppose  there is a constant $c>0$ such that  
$\abs{a(x)} \leq 1/c$, $s(x,y) \leq 1/c$ for all $x$, $y \in [0,1]$ and 
\begin{equation} \label{eq:lower_bound_s} 
z_{lk} \inf_{x,y \in I_l \times I_k} s(x,y) \geq c \, z_{lk} 
\end{equation}   
for all $l$, $k \in \db{K}$.

\item \label{assum:s_a_piecewise_continuous} \stepcounter{assumptions}
 \emph{Piecewise $1/2$-Hölder continuous  profile functions: }   
For all $l$, $k \in \db{K}$ the restrictions
$s|_{I_l \times I_k}$ and $a|_{I_l}$ have $1/2$-Hölder continuous extensions to the closed sets $\overline{I_l \times I_k}$ 
and $\overline{I_l}$, respectively. 

\item \emph{Piecewise Hölder continuous  deformation profile:} \label{assum:a_better_than_hoelder_1_2} \stepcounter{assumptions}
The restrictions $a|_{I_l}$ of the profile function $a \colon [0,1] \to \C$ have a $\theta$-Hölder continuous extension to $\overline{I_l}$  for each $l \in \db{K}$, where $\theta \in (1/2,\infty)$. 
\end{enumerate}

The constants in the assumptions \ref{assum:independent_centered} -- \ref{assum:a_better_than_hoelder_1_2} are  model parameters and  independent of $n$ and, therefore, the respective estimates are uniform in $n$.

The next theorem shows that the empirical spectral distribution of non-Hermitian random matrices with independent entries, a variance profile and a diagonal expectation has a deterministic limit as the matrix size tends to infinity. 
The independence of the limit from the entry distributions was shown in \cite[Appendix~C]{tao2010} 
and \cite[Theorem~1.3]{smallestsingularvalue_complex}. 
When $X$ is a Ginibre matrix, the convergence of the ESD was proved in 
\cite[Theorem~6]{Sniady2002} and for an $X$ with i.i.d.\ entries, in \cite[Theorem~1.17]{tao2010}.

\begin{theorem}[Convergence of empirical spectral distribution] \label{thm:global_law} 
Let the functions $s \colon [0,1]^2 \to [0,\infty)$ and $a \colon [0,1] \to \C$ satisfy  \ref{assum:primitive_upper_lower}, \ref{assum:s_a_piecewise_continuous} and \ref{assum:a_better_than_hoelder_1_2}.  
Then there exists a unique probability measure $\sigma$ on $\C$ such that the following holds. 
For each $n \in \N$, let $X_n = (x_{ij}^{(n)})_{i,j \in \db{n}} \in \C^{n\times n}$ be a random matrix and  $A_n =(a^{(n)}_{ij})_{i,j=1}^n$ deterministic, satisfying 
\ref{assum:independent_centered} and \ref{assum:smallest_singular_value}  and \eqref{eq:a_ij_and_s_ij_relation_a_and_s}. Then  
the empirical spectral distribution $\frac{1}{n} \sum_{\zeta \in \spec(X_n + A_n)} \delta_\zeta$ 
converges to $\sigma$ 
weakly in probability as $n \to \infty$, i.e.\ 
for every bounded, continuous function $f \colon \C \to \C$ 
and $\eps >0$, we have 
\[ \lim_{n \to \infty} \P \bigg( \absbb{ \frac{1}{n} \sum_{\zeta \in \spec(X_n + A_n)} f(\zeta) - \int_{\C} 
f(\zeta) \sigma(\dd\zeta) } > \eps \bigg) = 0. \] 
Here the sum $\sum_{\zeta \in \spec(X_n + A_n)}$ is over all eigenvalues of $X_n +A_n$, counted with multiplicity. 
\end{theorem} 

 The probability measure $\sigma$ depends only on $s$ and $a$ as we will see in the proof.   
The \hyperlink{proof:thm:global_law}{proof of Theorem~\ref{thm:global_law}} is presented in Section~\ref{subsec:proof_global_law} below. 
 We note that \ref{assum:smallest_singular_value} holds if $X = A \odot Y$ is the Hadamard product of 
a derministic matrix $A \in [0,1]^{n\times n}$ and a matrix $Y \in \C^{n\times n}$ with i.i.d\ entries 
of mean zero, unit variance and a finite $4 + \eta$-moment for some $\eta>0$ due to \cite[Theorem~1.17 and Remark~1.2]{Cook2018singularValue}.  
 Assumption \ref{assum:a_better_than_hoelder_1_2} requires stronger regularity of $a$ than Assumption~\ref{assum:s_a_piecewise_continuous} and is used in our proof to ensure that the image $a([0,1]) \subset \C$  of $a$ has Lebesgue measure zero.

\begin{definition} \label{def:limiting_spectral_measure} 
The probability measure $\sigma$ from Theorem~\ref{thm:global_law} is called \emph{limiting spectral measure} 
associated with $s$ and $a$. 
\end{definition} 

\begin{remark}
The limiting spectral measure $\sigma$ equals the Brown measure of a deformed  $L^\infty[0,1]$-valued circular element within the framework of operator-valued free probability theory. This identification follows from Proposition~\ref{pro:existence_sigma_and_sigma_n} below and \cite[Propositions 4.2 and D.1]{AK_Brown}.  The measure $\sigma$
has a bounded density with respect to the Lebesgue measure on $\C$, which has a real analytic extension to a neighbourhood of $\supp \sigma$ and the boundary of $\supp \sigma$ is an at most one-dimensional real analytic variety
 \cite[Theorem~2.2]{AK_Brown}. The density of $\sigma$ is strictly positive on $\supp \sigma$ apart from its singularities which are fully characterised in \cite[Theorem~2.5]{AK_Brown}. 
We refer to \cite[Section~3]{AK_Brown} for a few examples and figures depicting the shape of $\supp \sigma$ and comparing it to some sampled eigenvalues. 
 \end{remark}

The next theorem states that the pseudospectrum of the $n \times n$-matrix $X + A$ is asymptotically given by the support
of the measure $\sigma$ from Theorem~\ref{thm:global_law} which  coincides with the spectrum of $X + A$ by Theorem~\ref{thm:global_law} in the limit $n \to \infty$. 
We first introduce the pseudospectrum of a matrix. 
For any $\eps>0$, the \emph{$\eps$-pseudospectrum} of a matrix $R \in \C^{n \times n}$ is defined as the set 
\bels{eps speudospectrum}{
 \spec_\eps(R) :=\{ \zeta \in \C \colon \norm{(R - \zeta)^{-1} } \geq \eps^{-1} \} . 
}
Note that $\spec_\eps(R)$ is monotonically increasing in $\eps$ and $\spec(R) = \cap_{\eps >0} \spec_\eps(R)$. 

Furthermore, 
for a sequence $(\Omega_n)_{n \in \N}$ of sets we use the customary definitions  
\[  
\liminf_{n \to \infty} \Omega_n := \bigcup_{N\in \N} \bigcap_{n \geq N} \Omega_n, 
\qquad \qquad  
 \limsup_{n \to \infty} \Omega_n := \bigcap_{N \in \N} \bigcup_{n \geq N} \Omega_n.  
\]

\begin{theorem}[Spectrum occupies pseudospectrum]\label{thr:Spectrum occupies pseudospectrum}
Let $s \colon [0,1]^2 \to [0,\infty)$ and $a \colon [0,1] \to \C$ satisfy  \ref{assum:primitive_upper_lower}  and \ref{assum:s_a_piecewise_continuous}.  
For each $n \in \N$, let $X_n = (x_{ij}^{(n)})_{i,j \in \db{n}} \in \C^{n\times n}$ be a random matrix satisfying 
\ref{assum:independent_centered} and let $A_n = (a_{ij}^{(n)})_{i,j \in \db{n}} \in \C^{n\times n}$ be 
a deterministic matrix. 
If \eqref{eq:a_ij_and_s_ij_relation_a_and_s} holds for all $i$, $j \in \db{n}$ then  
  there exists a monotonically increasing family $(\spec_\eps^\infty(s, a))_{\eps>0}$ of deterministic subsets  of $\C$  such that, almost surely\footnote{We assume that all $X_n$ for $n \in\N$ are realised on the same probability space.}, 
\bels{convergence of pseudospectrum}{   \limsup_{n\to \infty} \spec_\eps(X_n + A_n) \subset  \spec_{\eps}^\infty(s, a) \subset{\liminf_{n \to \infty} \spec_{\eps+\delta}(X_n + A_n)}}
hold for all $\eps,\delta>0$. 
Moreover, this family is right continuous, i.e.\  
$\cap_{\delta>0}\spec_{\eps +\delta}^\infty(s, a)=\spec_{\eps}^\infty(s, a)$ and the limiting spectral measure $\sigma$ from Theorem~\ref{thm:global_law} satisfies
\begin{equation} \label{eq:cap_spec_eps_infty_equals_supp_sigma} 
 \bigcap_{\eps >0} \spec_{\eps}^\infty(s, a) = \supp \sigma. 
\end{equation} 
\end{theorem} 

The \hyperlink{proof:thr:Spectrum occupies pseudospectrum}{proof of Theorem~\ref{thr:Spectrum occupies pseudospectrum}} is given in Section~\ref{subsec:proof_theorem_spectrum_pseudospectrum} below. 
We note that the sets $(\spec_{\eps}^\infty(s,a))_{\eps>0}$ are monotonically increasing in $\eps$. 
In particular, Theorem~\ref{thr:Spectrum occupies pseudospectrum} implies that $\spec(X_n + A_n)$ is eventually 
almost surely contained in a neighbourhood of $\supp \sigma$.

In the independent work \cite{CampbellCipolloniErdosJi2024} the authors point out that for strictly positive variance profile function $s$ the inclusions \eqref{convergence of pseudospectrum} can be improved to a statement, where $\eps =\eps_n$ and $\delta  = \delta_n$ depend on $n$ with an optimal convergence rate by using the argument  in \cite[Remark 2.15]{CampbellCipolloniErdosJi2024} and using \cite[Proposition~5.16 (iv)]{AK_Brown} to verify one of their  assumptions.

\begin{remark}[Special cases for $\supp \sigma$]
In the  case when the entries of the random matrix $X_n$ are independent and identically distributed, i.e.\ when $s=t$ is a constant and $\E\1 \abs{x_{ij}}^2 = \frac{t}{n}$, 
 the well-known formula  
\begin{equation} \label{eq:formula_S_idd_case} 
\supp \sigma = \cbb{\zeta \in \C: \int_0^1\frac{\dd x}{\abs{a(x)-\zeta}^2} \geq\frac{1}{t}}\,
\end{equation} 
holds \cite{Khoruzhenko1996,tao2010}. 

 If $a = 0$ and the entries of $X_n$ are centered and independent with variances $s_{ij}= \E\abs{x_{ij}}^2$, then 
 $\supp \sigma = \{ \zeta \in \C \colon \abs{\zeta}^2 \leq \varrho(S) \}$, 
where $\varrho(S)$ denotes the spectral radius of $S =(s_{ij})_{i,j=1}^n$  \cite{Altcirc,Cook2022}. 
\end{remark}

\subsection{Notations}

We now introduce some notations used throughout. 
We write $\qq{n} \defeq \{ 1, \ldots, n\}$ for $n \in \N$. 
For $r >0$, we denote by $\mathbb{D}_r \defeq \{ z \in \C \colon \abs{z} < r \}$ the disk of radius $r$ around the origin in $\C$ and 
by $\dist(x,A) :=\inf\{ \abs{x-y} \colon y \in A \}$ the Euclidean distance of a point $x \in \C$ from a set $A \subset \C$. 

We use the convention that $c$ and $C$ denote  generic constants that may depend on the model parameters, but are otherwise     
 uniform in all other parameters, e.g.\ $n$, $\zeta$, etc.. 
For two real scalars $f$ and $g$ 
we write $f \lesssim g$ and $g \gtrsim f$ if  $f \leq C g$ for such a constant $C>0$. 
In case $f \lesssim g$ and $f \gtrsim g$ both hold,  we write $f \sim g$. 
If the constant $C$ depends on a parameter $\delta$ that is not a model parameter, we write $\lesssim_\delta$, $\gtrsim_\delta$ and 
$\sim_\delta$, respectively. 
The notation for inequality up to constant is also used for self-adjoint matrices/operators $f$ and $g$, where $f \leq C g$ is interpreted in the sense of  quadratic forms. 
For complex $f$  and $g \geq 0$  we write $f = O(g)$ in case $\abs{f} \lesssim g$. 
Analogously $f = O_\delta(g)$ expresses the fact  $\abs{f}\lesssim_\delta g$.

\section{Dyson equations and limiting measures} 

The purpose of this section is the construction of the limiting measure $\sigma$ from Theorem~\ref{thm:global_law} and the preparations of the proofs of the main results in the next section. 
We set $\mathcal B := L^\infty[0,1]$, where $[0,1]$ is equipped with the Lebesgue measure.  

In order to construct $\sigma$, given $a \in \mathcal B$ and a measurable function $s \colon [0,1]^2 \to [0,\infty)$, we consider two coupled equations for functions in $v_1, v_2 \in \cal{B}$ with $v_1>0$ and $v_2>0$, namely  
\begin{subequations} \label{eq:V_equations} 
\begin{align} 
\frac{1}{ v_1} & =  \eta +Sv_2+  \frac{\abs{\zeta-a}^2}{\eta +S^*v_1} \,, \label{eq:v1} \\ 
\frac{1}{ v_2} & =  \eta +S^*v_1+  \frac{\abs{\zeta-a}^2}{\eta +Sv_2}\,, \label{eq:v2} 
\end{align}
\end{subequations} 
for all $\eta >0$ and $\zeta \in \C$. 
The equation \eqref{eq:V_equations} is called the \emph{(vector) Dyson equation}. 
Here, $\eta$ and $\zeta$ are interpreted as the constant functions on $[0,1]$ with the respective value. 
Moreover, we assumed 
\[ \sup_{x \in [0,1]} \int_{0}^1 s(x,y) \dd y < \infty, 
\qquad \qquad \sup_{y \in [0,1]} \int_{0}^1 s(x,y) \dd x < \infty \]  
so that the two operators $S \colon \cal{B} \to \cal{B} $ and $S^* \colon \cal{B} \to  \cal{B}$ 
defined through 
\begin{equation} \label{eq:def_operators_S_and_S_star} 
(S u)(x) = \int_{0}^1 s(x,y) u(y) \dd y, \qquad (S^* u)(x) = \int_{0}^1 s(y,x) u(y) \dd y 
\end{equation} 
for all $x \in [0,1]$ and $u \in \cal{B}$ are well-defined 
bounded linear operators.  
The existence and uniqueness of $(v_1,v_2)$ is established in Lemma~\ref{lem:existence_uniqueness} below 
by relating it to a matrix-valued version of \eqref{eq:V_equations}.

\subsection{Relation to Matrix Dyson equation} \label{sec:relation_dyson_mde} 

Let $(v_1, v_2)$ be a solution of \eqref{eq:V_equations}. 
We now relate $(v_1,v_2)$ to a solution $M  \in \mathcal B^{2\times 2}$ of a matrix equation. 
To that end, we  introduce   
\begin{equation} \label{eq:structure_M} 
 y:=\frac{v_1\2( \bar a-\bar \zeta)}{\eta + S^*v_1}, \qquad 
 M := \begin{pmatrix} \ii v_1 & \ol{y} \\ y & \ii v_2 \end{pmatrix}\in \mathcal B^{2\times 2}.
\end{equation} 
Then  $\Im M := \frac{1}{2\ii} (M - M^*)$ is positive definite
 and inverting the $2\times 2$ matrix $M$ explicitly shows that $M$ satisfies the \emph{Matrix Dyson Equation (MDE)} 
\begin{equation} \label{eq:mde} 
- M^{-1} =  \begin{pmatrix} \ii \eta & \zeta - a \\ \overline{\zeta- a} & \ii \eta \end{pmatrix} 
+ \Sigma[M]. 
\end{equation} 
Here, $\Sigma \colon \mathcal B^{2\times 2} \to \mathcal B^{2\times 2}$ is defined through 
\begin{equation} \label{eq:def_Sigma} 
 \Sigma \bigg[ \begin{pmatrix} r_{11} & r_{12} \\ r_{21} & r_{22} \end{pmatrix} \bigg] 
= \begin{pmatrix} S r_{22} & 0 \\ 0 & S^* r_{11} \end{pmatrix} 
\end{equation} 
for all $r_{11}$, $r_{12}$, $r_{21}$, $r_{22} \in \mathcal B$.  
In order to see that $M$ from \eqref{eq:structure_M} satisfies \eqref{eq:mde}, we note that 
$y=\frac{v_2\2(\bar a-\bar \zeta)}{\eta + Sv_2}\,$, which is a consequence of 
$v_2(\eta + S^*v_1)=v_1(\eta + Sv_2)$. The latter follows directly from \eqref{eq:V_equations}. 
On the other hand, if $M \in \mathcal B^{2\times 2}$ with $\Im M$ positive definite is a solution of \eqref{eq:mde} 
then it is easy to see that denoting the diagonal elements of $M$ by $\ii v_1$ and $\ii v_2$ yields a solution of \eqref{eq:V_equations}.

When studying the ESD of the $n\times n$ random matrix $X + A$, it will be convenient to compare it first to a deterministic, $n$-dependent measure. 
The initial step for the definition of this measure is the following discretised version of \eqref{eq:V_equations}. 
Given $a \in \C^n$ and a matrix $S^{(n)} \in [0,\infty)^{n\times n}$ with positive entries, we also consider 
the $n$-dependent vector Dyson equation 
\begin{equation} \label{eq:V_equations_n}  
\frac{1}{v_1^{(n)}} = \eta + S^{(n)} v_2^{(n)} + \frac{\abs{a^{(n)} - \zeta}^2}{\eta + (S^{(n)})^* v_1^{(n)}}, 
\qquad 
 \frac{1}{v_2^{(n)}} = \eta + (S^{(n)})^* v_1^{(n)} + \frac{\abs{a^{(n)} - \zeta}^2}{\eta + S^{(n)} v_2^{(n)}}  
\end{equation} 
with $v_1^{(n)}$, $v_2^{(n)} \in (0,\infty)^n$.  
The existence and uniqueness of solutions to \eqref{eq:V_equations} and \eqref{eq:V_equations_n} is 
the content of the next lemma. 

\begin{lemma}[Existence and uniqueness of solutions] \label{lem:existence_uniqueness} 
For each $\eta >0$ and $\zeta \in \C$, the equations \eqref{eq:V_equations} 
and \eqref{eq:V_equations_n} have unique solutions $(v_1,v_2) \in \mathcal B_+^2$ and $(v_1^{(n)}, v_2^{(n)}) 
\in (0,\infty)^{2n}$, respectively. 
\end{lemma} 

\begin{proof} 
Owing to the identification of solutions to \eqref{eq:V_equations} and \eqref{eq:mde} in Section~\ref{sec:relation_dyson_mde}, 
we can now infer the existence and uniqueness of the solution to \eqref{eq:V_equations} to the existence and 
uniqueness of the solution to \eqref{eq:mde}. 
Indeed, the latter is a  simple case of the general existence and uniqueness result 
\cite[Theorem~2.1]{HeltonRashidiFarSpeicher2007}. 
The existence and uniqueness of the solution to \eqref{eq:V_equations_n} follow analogously by introducing a matrix equation 
on $\C^{2n \times 2n}$ analogously to \eqref{eq:mde} (see \eqref{eq:mde_general_n} below with $w = \ii \eta$) 
and invoking 
\cite[Theorem~2.1]{HeltonRashidiFarSpeicher2007}. 
This proves Lemma~\ref{lem:existence_uniqueness}. 
\end{proof}

\paragraph{Matrix Dyson equation with general spectral parameter and measure $\rho_\zeta$} 
 The definition of the sets $\spec_\eps^\infty(s,a)$ from Theorem~\ref{thr:Spectrum occupies pseudospectrum} requires the spectral parameter $\ii \eta$ in 
the MDE, \eqref{eq:mde}, to be replaced by a general 
$w \in \C$ with $\Im w >0$. Given such $w$, we consider the MDE  
\begin{equation} \label{general MDE} 
- M(\zeta,w)^{-1} = \begin{pmatrix} w & \zeta - a \\  \ol{\zeta - a} & w \end{pmatrix} + \Sigma[M(\zeta,w)] 
\end{equation} 
for $\zeta \in \C$. 
Then \eqref{general MDE} has a unique solution $M(\zeta, w) \in \mathcal B^{2\times 2}$ under the constraint that $\Im M(\zeta, w) := \frac{1}{2\ii} (M(\zeta, w) - M(\zeta, w)^*)$ is 
positive definite for $\Im w > 0$ \cite[Theorem~2.1]{HeltonRashidiFarSpeicher2007}. 

By \cite[Proposition~2.1 and Definition~2.2]{AEK_Shape}, the map $w \mapsto \avg{M(\zeta, w)}$ is the Stieltjes transform of a probability measure on $\R$, where we introduced the short hand notation 
\[
\avg{R}: = \frac{1}{2}(\avg{r_{11}} + \avg{r_{22}}) \,, \qquad R= \mtwo{r_{11} & r_{12}}{r_{21} & r_{22} } \in \cal{B}^{2 \times 2}\,.
\] 
The measure $\rho_\zeta$ introduced in the next definition will turn out to be the limiting measure of 
the symmetrized singular value distribution of $X_n + A_n - \zeta$ for any $\zeta \in \C$, 
see \cite[Theorem~2.7]{AEKN_Kronecker} and Corollary~\ref{cor:singular_value_density_n_to_rho_zeta} below.

\begin{definition} \label{def:rho_zeta} 
We denote by $\rho_\zeta$ the unique probability measure on $\R$ whose Stieltjes transform is given by $w \mapsto \avg{M(\zeta, w)}$. 
 We call $\rho_\zeta$ the limiting singular value measure.  
 Through $\rho_\zeta$ we define
\bels{eq:def_S_eps}{
 \mathbb S_\eps:= \{\zeta \in \C: \dist(0, \supp \rho_\zeta)\le  \eps\}
}
for any $\eps\geq 0$.
\end{definition}

The set $\spec_\eps^\infty(s,a)$ from Theorem~\ref{thr:Spectrum occupies pseudospectrum} is identified 
with $\mathbb S_\eps$ from \eqref{eq:def_S_eps} in the \hyperlink{proof:thr:Spectrum occupies pseudospectrum}{proof of Theorem~\ref{thr:Spectrum occupies pseudospectrum}} (see \eqref{eq:def_spec_eps_infty_s_a} below). 
This set is the $n \to \infty$ limit of the $\eps$-pseudospectrum  \eqref{eps speudospectrum}  for $R=X_n+A_n$.

\begin{remark}\label{rmk:S eps}
The sets $\mathbb S_\eps$ defined in \eqref{eq:def_S_eps} are 
monotonously nondecreasing in $\eps\geq 0$, i.e.\  $\mathbb S_{\eps_1} \subset \mathbb S_{\eps_2}$ if $\eps_1 \le \eps_2$. 
Moreover, they are bounded, in fact, $\mathbb S_{\eps} \subset \{\zeta \in \C: \abs{\zeta} \le \eps + \norm{a}_\infty + 2 (\norm{S}_{\infty})^{1/2}\}$ for all $\eps \geq 0$ as a consequence of \cite[Proposition~2.1]{AEK_Shape}.
 Here, $\norm{S}_\infty$ denotes the operator norm of $S$ viewed as an operator from $\mathcal B$ to $\mathcal B$.  
\end{remark}

\paragraph{Bounds on $v$ and $v^{(n)}$} 

Throughout the following, given $a \in \mathcal B$ and $s \colon [0,1]^2 \to [0,\infty)$, we set 
\begin{equation} \label{eq:S_n_a_n_discretised} 
S^{(n)} := \bigg(\frac{1}{n} s\bigg(\frac{i}{n},\frac{j}{n}\bigg)\bigg)_{i,j \in \db{n}} \in [0,\infty)^{n\times n}, \qquad a^{(n)} := \bigg(a\bigg(\frac{i}{n} \bigg)\bigg)_{i \in \db{n}} \in \C^n. 
\end{equation} 
and consider $(v_1^{(n)}, v_2^{(n)})$ the solution of \eqref{eq:V_equations_n} with these choices of $S^{(n)}$ and 
$a^{(n)}$. 

\begin{lemma}[Bounds on $\avg{v_1}$] \label{lem:v_bounds}
Let  $a \in \mathcal B$ and  $s$ satisfy \ref{assum:primitive_upper_lower}. 
Then, uniformly for $\zeta \in \C$ 
and $\eta>0$, we have 
\begin{equation} \label{eq:avg_M_bounded} 
 0 \leq \avg{v_1(\zeta,\eta)} \lesssim \frac{1}{1 + \eta}. 
\end{equation} 

Furthermore, uniformly for any $T >0 $ and $\zeta \in \C$, we have 
\begin{equation} \label{eq:integral_bounds} 
\int_0^T \absbb{\avg{v_1 (\zeta, \eta)} - \frac{1}{1 + \eta} } \dd \eta 
\lesssim \min \Big\{ T , \sqrt{1 + \abs{\zeta}} \Big\}, 
\qquad 
 \int_T^\infty \absbb{\avg{v_1(\zeta,\eta)} - \frac{1}{1 + \eta} } \dd \eta 
\lesssim \frac{1+\abs{\zeta}}{T}. 
\end{equation}  
The same estimates hold uniformly for $n \in \N$ when $v_1$ is replaced by $v_1^{(n)}$ from \eqref{eq:V_equations_n}, 
where $\avg{v} = \frac{1}{n} \sum_{i\in \db{n}} v_i$ for $v = (v_i)_{i \in \db{n}} \in \C^n$.  
\end{lemma} 

\begin{proof} 
Owing to \ref{assum:primitive_upper_lower}, the bounds on $\avg{v_1(\zeta,\eta)}$ follow directly from  
\cite[Lemma~6.1]{AK_Brown}.  
For $\avg{v_{1}^{(n)}(\zeta,\eta)}$, we note that 
 \cite[Assumption~A1]{AK_Brown}  
 is satisfied with $\mathfrak X = \db{n}$ and $\mu$ the counting measure on $\mathfrak X$ due to \ref{assum:primitive_upper_lower} and \eqref{eq:S_n_a_n_discretised}. 
Therefore, the bounds on $\avg{v_{1}^{(n)}(\zeta,\eta)}$ also follow from   
\cite[Lemma~6.1]{AK_Brown}. 
\end{proof}

\subsection{Limiting spectral measure and its support} 

In this subsection, we introduce the limiting spectral measure $\sigma$ and an $n$-dependent measure $\sigma^{(n)}$ 
that will turn out to approximate $\sigma$ when $n$ tends to infinity. 

The bounds in \eqref{eq:integral_bounds} for $v_1$ and $v_1^{(n)}$ imply that, for each $\zeta \in \C$, 
the integrals 
\begin{equation} \label{eq:def_L} 
 L(\zeta) := \int_0^\infty \bigg(\avg{v_1(\zeta, \eta)} - \frac{1}{1 + \eta} \bigg) \dd \eta, \qquad 
 L^{(n)}(\zeta) := \int_0^\infty \bigg(\avg{v_1^{(n)}(\zeta, \eta)} - \frac{1}{1 + \eta} \bigg) \dd \eta 
\end{equation} 
exist in the Lebesgue sense, where $v_1$ and $v_1^{(n)}$ are the unique solutions of \eqref{eq:V_equations} 
and \eqref{eq:V_equations_n}, respectively. 

In the proof of Theorem~\ref{thm:global_law} in Section~\ref{subsec:proof_global_law} below, we relate this definition to the limiting measure 
of the empirical spectral distribution. In particular, we refer to \eqref{eq:spectral_statistics_log_determinant}, 
\eqref{eq:log_determinant_integral} and Proposition~\ref{pro:global_law_H} below.  
 The following two propositions recall results from \cite{AK_Brown} about how to define the limiting spectral measure $\sigma$ in terms of the Laplacian of $L$ from \eqref{eq:def_L} as well as about its density and  support, respectively. For the proof of these statements we point to the corresponding results from \cite{AK_Brown}.  

\begin{proposition}[Existence of $\sigma$ and $\sigma^{(n)}$] 
\label{pro:existence_sigma_and_sigma_n} 
If $a \in \mathcal B$, $s$ satisfies \ref{assum:primitive_upper_lower} 
and $a^{(n)}$ and $S^{(n)}$ are chosen as in \eqref{eq:S_n_a_n_discretised} 
then the following holds. 
\begin{enumerate}[label=(\roman*)] 
\item \label{item:L_sigma} 
There is a unique probability measure $\sigma$ on $\C$ such that 
\begin{equation} \label{eq:sigma_L_identity} 
 \int_{\C} f(\zeta ) \sigma(\dd \zeta) = -\frac{1}{2\pi} \int_{\C} \Delta f(\zeta) L(\zeta ) \dd^2 \zeta 
\end{equation} 
for all $f \in C_0^2 (\C)$, where $\dd^2 \zeta$ denotes the Lebesgue measure on $\C$.  
\item \label{item:L_n_sigma_n} 
For each $n \in\N$, there is a unique probability measure $\sigma^{(n)}$ on $\C$ such that 
\eqref{eq:sigma_L_identity} holds when $\sigma$ and $L$ are replaced by $\sigma^{(n)}$ and $L^{(n)}$ from \eqref{eq:def_L}, respectively. 
Furthermore, there is $\varphi \sim 1$ such that $\supp \sigma^{(n)} \subset \mathbb{D}_\varphi$ for all $n \in \N$.
\end{enumerate} 
\end{proposition} 

\begin{proof} 
 Part \ref{item:L_sigma} is a direct consequence of \cite[Proposition~4.2(i)]{AK_Brown} because of \ref{assum:primitive_upper_lower}. 
Clearly, $a^{(n)}$ and $S^{(n)}$ from \eqref{eq:S_n_a_n_discretised} satisfy 
 $\norm{a^{(n)}}_\infty \lesssim 1$  and  \cite[Assumption~A1]{AK_Brown}  (with the same constants as $a$ and $s$).  
Hence,  \cite[Lemma~6.2(i), Proposition~4.2(i), Corollary~6.4]{AK_Brown}  and Remark~\ref{rmk:S eps}  
imply the well-definedness of $L^{(n)}$ and the existence of probability measures $\sigma^{(n)}$ satisfying 
\eqref{eq:sigma_L_identity} for $L^{(n)}$ as well as $\supp \sigma^{(n)} \subset  \mathbb{D}_\varphi$, respectively.   
\end{proof}

\begin{proposition}[Density and support of $\sigma$] \label{pro:limiting_measure_and_S} 
Let $a$ and $s$ satisfy \ref{assum:primitive_upper_lower} and \ref{assum:s_a_piecewise_continuous}. 
Then the measure $\sigma$ from Proposition~\ref{pro:existence_sigma_and_sigma_n} 
has a bounded density 
with respect to the Lebesgue measure $\dd^2\zeta$ on $\C$, i.e.\ 
$\sigma(\dd \zeta) = \sigma(\zeta) \dd^2 \zeta$ for some bounded, measurable function $\sigma \colon \C\to [0,\infty)$. 
Moreover, 
\begin{equation} \label{eq:overline_S_equal_intersection_S_eps} 
\supp \sigma =  \mathbb S_0 . 
\end{equation} 
\end{proposition} 

\begin{proof}  
The existence of the density is stated in \cite[Theorem~2.2(i)]{AK_Brown}. 
The identity \eqref{eq:overline_S_equal_intersection_S_eps} follows from \cite[Theorem~2.2(i) and (6.33)]{AK_Brown}. 
\end{proof}

\subsection{Approximating spectral measure and singular value measure} 

Throughout the following, $S^{(n)}$ and $a^{(n)}$ are chosen as in \eqref{eq:S_n_a_n_discretised}
and, with these choices, $v_1^{(n)}$, $v_2^{(n)}$, $L^{(n)}$ and $\sigma^{(n)}$ are the associated objects from \eqref{eq:V_equations_n},
 \eqref{eq:def_L} and Proposition~\ref{pro:existence_sigma_and_sigma_n} \ref{item:L_n_sigma_n}, respectively. 

The next corollary states the promised convergence of $\sigma^{(n)}$ to $\sigma$. It follows readily from 
Lemma~\ref{lem:discretizing_dyson_eq} below. We will present the detailed \hyperlink{proof:cor:convergence_sigma_n_to_sigma}{proof of Corollary~\ref{cor:convergence_sigma_n_to_sigma}} in Section~\ref{sec:discretizing_Dyson_equation} 
below. 

\begin{corollary} \label{cor:convergence_sigma_n_to_sigma} 
If $s$ and $a$ satisfy \ref{assum:s_a_piecewise_continuous} and \ref{assum:primitive_upper_lower} then 
$\sigma^{(n)}$ converges to $\sigma$ weakly as $n$ tends to infinity. 
\end{corollary}

Similarly, we now introduce an $n$-dependent approximation of the limiting singular value measure 
$\rho_\zeta$ from Definition~\ref{def:rho_zeta}. First, we write up the matrix Dyson equation, \eqref{general MDE}, 
in the $n$-dependent setup. 
For $S^{(n)}$ and $a^{(n)}$ as in \eqref{eq:S_n_a_n_discretised}, $\zeta \in \C$, 
$w \in \C$ with $\Im w>0$ and $M^{(n)} \in \C^{2n\times 2n}$, we consider 
\begin{equation} \label{eq:mde_general_n} 
- M^{(n)}(\zeta,w)^{-1} = \begin{pmatrix} w & \zeta - a^{(n)}\\  \ol{\zeta - a^{(n)}} & w \end{pmatrix} + \Sigma^{(n)}[M^{(n)}(\zeta,w)]. 
\end{equation} 
Here, the linear map $\Sigma^{(n)} \colon \C^{2n\times 2n} \to \C^{2n\times 2n}$ is defined through 
\[ \Sigma^{(n)} \bigg[ \begin{pmatrix} R_{11} & R_{12} \\ R_{21} & R_{22} \end{pmatrix} \bigg] = \begin{pmatrix} S^{(n)} r_{22} & 0 \\ 0 & (S^{(n)})^* r_{11} \end{pmatrix}, \] 
where $R_{11}$, $R_{12}$, $R_{21}$, $R_{22} \in \C^{n\times n}$ and $r_{11} = \diag(R_{11})$, $r_{22} = \diag(R_{22})$ 
are the diagonals of $R_{11}$ and $R_{22}$ interpreted as vectors in $\C^n$. 
Then $S^{(n)} r_{22}$ and $(S^{(n)})^*r_{11}$ are identified with the diagonal matrices, whose diagonals are the vectors $S^{(n)} r_{22}$ and $(S^{(n)})^*r_{11}$, respectively. 
Under the constraint that $\Im M^{(n)} = \frac{1}{2\ii} (M^{(n)} - (M^{(n)})^*)$ is positive definite, 
\eqref{eq:mde_general_n} has a unique solution by \cite[Theorem~2.1]{HeltonRashidiFarSpeicher2007}.

Throughout the following, we denote by $M^{(n)}$ the unique solution of \eqref{eq:mde_general_n} with $S^{(n)}$ and $a^{(n)}$ as in \eqref{eq:S_n_a_n_discretised}. 
Let $\rho_\zeta^{(n)}$ be the probability measure on $\R$, whose Stieltjes transform is given 
by  $w \mapsto \frac{1}{2n}\tr M^{(n)}(\zeta, w)$, i.e.\ 
\begin{equation} \label{ST of rho n} 
\int_{\R} \frac{\rho_\zeta^{(n)} (\dd x)}{x - w} = \frac{1}{2n} \tr M^{(n)}(\zeta,w) 
\end{equation} 
for all $w \in \C$ with $\Im w>0$. Here, $\tr$ denotes the trace on $\C^{2n\times 2n}$.
Next, we relate $\rho_\zeta^{(n)}$ and $\rho_\zeta$. 

\begin{corollary} \label{cor:singular_value_density_n_to_rho_zeta} 
If $s$ and $a$ satisfy  \ref{assum:s_a_piecewise_continuous}  
 then the following holds for each fixed $\zeta \in \C$. 
\begin{enumerate}[label=(\roman*)] 
\item \label{item:convergence_rho_zeta_n} 
$\rho_\zeta^{(n)}$ converges to $\rho_\zeta$ weakly as $n$ tends to infinity. 
\item  \label{item:limsup_rho_zeta_n_subset_supp_rho_zeta} 
 For each $\delta>0$, $\supp \rho_\zeta^{(n)} \subset \supp \rho_\zeta + (-\delta,\delta)$ for all sufficiently large $n \in \N$. 
In particular, $ \limsup_{n \to \infty}\supp \rho_\zeta^{(n)} \subset \supp \rho_\zeta $.  
\end{enumerate} 
\end{corollary}

Corollary~\ref{cor:singular_value_density_n_to_rho_zeta} will also be derived from Lemma~\ref{lem:discretizing_dyson_eq} below. 
The \hyperlink{proof:cor:singular_value_density_n_to_rho_zeta}{proof of Corollary~\ref{cor:singular_value_density_n_to_rho_zeta}} will be given in Section~\ref{sec:discretizing_Dyson_equation} below.

\section{Proof of main results -- Theorem~\ref{thm:global_law} and Theorem~\ref{thr:Spectrum occupies pseudospectrum}} 
\label{sec:proofs_main_results}

This section is devoted to the proofs of our main results, Theorem~\ref{thm:global_law} and Theorem~\ref{thr:Spectrum occupies pseudospectrum}. They are derived from the results in the previous sections as well as some 
inputs from \cite{AEKN_Kronecker,smallestsingularvalue,smallestsingularvalue_complex}. 
The underlying idea for both derivations is the Hermitization approach going back to Girko \cite{Girko1984} 
which allows to understand the eigenvalue density of $X + A$ by understanding the spectra 
of the Hermitian matrices $(H_\zeta)_{\zeta \in \C}$ defined through 
\begin{equation} \label{eq:def_H_zeta} 
H_\zeta := \begin{pmatrix} 0 & X+A - \zeta \\ (X+A - \zeta)^* & 0 \end{pmatrix}. 
\end{equation} 
The usefulness of $H_\zeta$ becomes apparent from the following properties. A complex number $\zeta\in \C$ is an eigenvalue 
of $X + A$ if and only if $H_\zeta$ has a nontrivial kernel. 
Furthermore, the spectrum of $H_\zeta$ is symmetric around zero and its non-negative eigenvalues coincide with the singular values of $X+A - \zeta$ (with multiplicities).

\subsection{Proof of Theorem~\ref{thm:global_law}} \label{subsec:proof_global_law} 

After this general explanation, we now focus on the proof of Theorem~\ref{thm:global_law}. 
To that end, we now explain in detail how the empirical spectral distribution of $X + A$ is 
expressed in terms of the family $(H_\zeta)_{\zeta \in \C}$.

First, as $\log \abs{\, \cdot\, }$ is the fundamental solution for the Laplace operator on $\C$, we obtain 
\begin{equation} \label{eq:spectral_statistics_log_determinant} 
\frac{1}{n} \sum_{\xi \in \spec(X+ A)} f(\xi) = \frac{1}{2\pi n} \sum_{\xi \in \spec(X + A)} 
\int_{\C} \Delta f(\zeta) \log \abs{\xi - \zeta} \dd^2 \zeta 
= \frac{1}{4\pi n} \int_{\C} \Delta f(\zeta) \log \abs{\det H_\zeta} \dd^2 \zeta, 
\end{equation} 
where the last step follows from 
\begin{equation} \label{eq:log_determinant} 
\sum_{\xi \in \spec(X + A)} \log \abs{\xi- \zeta} = \log \abs{\det (X + A -\zeta)} = \frac{1}{2} \log \abs{\det H_\zeta}. 
\end{equation} 
We can now express the log-determinant of $H_\zeta$ as an integral of the normalised trace of the 
resolvent $G(\zeta, \ii \eta) := (H_\zeta -\ii \eta)^{-1}$ of $H_\zeta$ on the imaginary axis; this expression reads as 
\begin{equation} \label{eq:log_determinant_integral} 
\log \abs{\det H_\zeta} = - 2 n \int_0^T \Im \avg{G(\zeta, \ii \eta)} \dd \eta + \log \abs{\det (H_\zeta - \ii T)} 
\end{equation} 
for any $T >0$ (see \cite{TaoVu2015} for an application of \eqref{eq:log_determinant_integral} in a similar context).  
Here and in the following, for a $K\times K$-matrix $R \in \C^{K\times K}$, we denote by 
 $\avg{R} = \frac{1}{K} \Tr R$ the normalized trace of $R$.

For the proof of Theorem~\ref{thm:global_law}, we follow the strategy of \cite[proof of Theorem~2.3]{AK_Corr_circ}, which is presented in \cite[Section~3.2]{AK_Corr_circ}. 

The next proposition, which follows directly from results in \cite{AEKN_Kronecker}, shows that $\avg{G(\zeta, \ii \eta)}$ 
is approximately deterministic. Given \eqref{eq:spectral_statistics_log_determinant} and 
\eqref{eq:log_determinant_integral}, 
this explains the origin of the definition of $\sigma$  via \eqref{eq:sigma_L_identity} and \eqref{eq:def_L}.  

In the next proposition and throughout this section, we use the following notion of high probability events. 
We say that a sequence of events $(\Omega_n)_{n \in \N}$ occurs \emph{with very high probability} if for each 
$\nu \in \N$, there is a constant $C_\nu >0$ (i.e.\ $C_\nu$ does not depend on $n$) such that 
$\P(\Omega_n) \ge 1- C_\nu n^{-\nu}$ for all $n \in \N$.

\begin{proposition}[Deterministic approximation of resolvent of $H_\zeta$, averaged version] \label{pro:global_law_H} 
Let $X \in \C^{n\times n}$ satisfy  \ref{assum:independent_centered}.  Let   $A = D(a^{(n)}) := (a_i^{(n)} \delta_{ij})_{i,j \in \db{n}}$ for some $a^{(n)}=(a_i^{(n)})_{i \in \db{n}} \in \C^n$ with  
 $\norm{a^{(n)}}_\infty=\max_{i\in\db{n}}\abs{a_i^{(n)}} \lesssim 1$.  
 Let $(v_1^{(n)},v_2^{(n)})$ be the solution of  \eqref{eq:V_equations_n} with $S^{(n)} = (\E\abs{x_{ij}}^2)_{i,j\in \db{n}}$.   
Let $\varphi >0$ be fixed. 
Then there are universal constants $\delta >0$ and $P \in \N$ such that 
\[ 
\abs{\avg{G(\zeta, \ii \eta)} - \ii \avg{v_1^{(n)}(\zeta, \eta)}} \leq \frac{n^{P\delta}}{(1 + \eta^2) n} 
\] 
with very high probability uniformly for all $n\in \N$, $\eta \in [n^{-\delta},\infty)$ and $\zeta \in  \mathbb{D}_\varphi$. 
\end{proposition}

\begin{proof}[\linkdest{proof:pro:global_law_H}{Proof}] 
The matrix $X + A$ is a Kronecker matrix according to \cite[Definition~2.1]{AEKN_Kronecker} 
with the choices $L =1$, $\ell =1$, $\wt{\alpha}_1 = 1$, $X_1 = X$, $\beta_1 = 0$, $Y_1 = 0$ and $\wt{a}_i =  a_i^{(n)}$ for all $i \in \db{n}$. 
In particular, the Hermitization $H_\zeta$ defined in \eqref{eq:def_H_zeta} is also a Kronecker matrix. 
Moreover, $H_\zeta$ satisfies the assumptions of \cite[Lemma~B.1 (ii)]{AEKN_Kronecker} due to \ref{assum:independent_centered} 
 and  $\norm{a^{(n)}}_\infty \lesssim 1$.  
Since the Hermitized matrix Dyson equation from \cite[eq.s~(2.2) -- (2.6)]{AEKN_Kronecker} 
coincides with the matrix Dyson equation,  \eqref{eq:mde_general_n}, associated with \eqref{eq:V_equations_n} for $(v_1^{(n)}, 
v_2^{(n)})$ and $S^{(n)} = (\E\abs{x_{ij}}^2)_{i,j\in \db{n}}$,  
 \cite[eq.~(B.5)]{AEKN_Kronecker} and \cite[eq.~(4.46)]{AEKN_Kronecker} imply Proposition~\ref{pro:global_law_H}. 
\end{proof}

The next lemma controls the number of small 
singular values of $X + A - \zeta$ and follows  from Proposition~\ref{pro:global_law_H} and an upper bound on $\abs{\avg{M(\zeta,\ii\eta)}}$.

\begin{lemma}[Number of small singular values of $X+ A - \zeta$] \label{lem:small_singular_values} 
 Let $a \in \mathcal B$ and $s$ satisfy \ref{assum:primitive_upper_lower}. 
Let $X$ satisfy \ref{assum:independent_centered} and let $A \in \C^{n\times n}$ be deterministic. 
Suppose that the entries of $X$ and $A$ satisfy \eqref{eq:a_ij_and_s_ij_relation_a_and_s}. 
Let $\varphi >0$ be fixed.  
Then there is a universal constant $\delta >0$ such that 
\[ \abs{ \spec(H_\zeta) \cap [-\eta,\eta]  } \lesssim n \eta \] 
with very high probability uniformly for all $\eta \in [n^{-\delta}, \infty)$ and $\zeta \in  \mathbb{D}_\varphi$. 
\end{lemma} 

\begin{proof} 
Proposition~\ref{pro:global_law_H} and  \eqref{eq:avg_M_bounded} for $\avg{v_1^{(n)}(\zeta,\eta)}$  
imply that the trace of $G(\zeta,\ii \eta)$ is bounded by a multiple of $n$ with very high probability. More precisely, $\abs{\Tr G(\zeta, \ii \eta)} \lesssim 
n$ with very high probability uniformly for all $\eta \in [n^{-\delta}, \infty)$ and $\zeta \in \mathbb{D}_\varphi$.
Hence, we set $\Sigma_\eta := \spec(H_\zeta)\cap [-\eta, \eta]$ and estimate 
\[ \frac{\abs{\Sigma_\eta}}{2 \eta} \leq \sum_{\lambda \in \Sigma_\eta} \frac{ \eta}{\lambda^2 + \eta^2} 
\leq \Im \Tr G(\zeta, \ii\eta) \lesssim n. \qedhere \]
\end{proof} 

We apply the previous results, i.e.\ Proposition~\ref{pro:global_law_H} and Lemma~\ref{lem:small_singular_values}, 
as well as \ref{assum:smallest_singular_value}    to 
the right-hand side of \eqref{eq:linear_statistics_reformulation_proof_global_law} by discretizing the integral in $\zeta$ through the next lemma. 

\begin{lemma}[Monte Carlo Sampling] \label{lem:monte_carlo} 
Let $\Omega \subset \C$ be bounded and of positive Lebesgue measure. Let $\mu$ be the normalised Lebesgue measure on $\Omega$ and $F \colon \Omega \to \C$ square-integrable with respect to $\mu$. 
Let $N\in \N$ and $\xi_1$, \ldots, $\xi_N$ be independent random variables distributed according to $\mu$. 
Then, for any $\eps >0$, we have 
\[ \P \bigg( \absbb{\frac{1}{N} \sum_{i=1}^N F(\xi_i) - \int_\Omega F \dd \mu} 
\leq \frac{1}{\sqrt{N \eps}} \Big( \int_\Omega \absB{F - \int_\Omega F \dd \mu}^2 \Big)^{1/2} 
\bigg) \geq 1 - \eps. \] 
\end{lemma} 

Lemma~\ref{lem:monte_carlo} is a special case of \cite[Lemma~36]{TaoVu2015}. For the convenience of the reader, we present the very short proof here.

\begin{proof} 
Each of the i.i.d.\ random variables $F(\xi_1)$, \ldots, $F(\xi_m)$ has expectation $\int_\Omega F\dd \mu$ 
and variance $\int_\Omega \abs{F - \int_\Omega F\dd \mu}^2 \dd\mu$. Hence, 
Chebysheff's inequality yields Lemma~\ref{lem:monte_carlo}.  
\end{proof}

The final ingredient for the proof of Theorem~\ref{thm:global_law} is the following remark which asserts that 
all eigenvalues of $X + A$ are contained in $\mathbb S_\eps$ defined in \eqref{eq:def_S_eps} with very high probability. 

\begin{remark}[No outlier eigenvalues of $X + A$] \label{rem:exclusion_eigenvalues} 
If $X$ satisfies  
\ref{assum:independent_centered} and $A=D(a)$ for some $a \in \C^n$ with $\norm{a}_\infty \lesssim 1$ then, 
for every $\eps>0$ and $\delta \in (0,\eps)$, 
all eigenvalues of $X + A$ are contained in $\mathbb{S}_{\eps}$ with very high probability, i.e.\ for each $\nu>0$, 
there is a constant $C \equiv C_{\eps,\delta, \nu}>0$ such that 
\[ \P \big( \spec(X + A) \subset \spec_{\eps-\delta}(X + A)  \subset \mathbb S_{\eps}\big) \geq 1 - C n^{-\nu} \] 
for all $n \in \N$. 
This follows directly from \cite[Lemma~6.1]{AEKN_Kronecker}  and Corollary~\ref{cor:singular_value_density_n_to_rho_zeta} \ref{item:limsup_rho_zeta_n_subset_supp_rho_zeta}.  Here, we used that $X + A$ is a Kronecker matrix 
according to \cite[Definition~2.1]{AEKN_Kronecker} 
and that the Dyson equation \eqref{eq:mde} and \cite[eq.~(2.6)]{AEKN_Kronecker} coincide  as explained in the \hyperlink{proof:pro:global_law_H}{proof of Proposition~\ref{pro:global_law_H}}.  
\end{remark}

We have now collected all ingredients for the proof of Theorem~\ref{thm:global_law}, which we present next. 

\begin{proof}[\linkdest{proof:thm:global_law}{Proof of Theorem~\ref{thm:global_law}}]  
 
Let $a^{(n)}$ and $S^{(n)}$ be defined as in \eqref{eq:S_n_a_n_discretised}. Given these choices, let $(v_1^{(n)},
v_2^{(n)})$ be the solution of \eqref{eq:V_equations_n} and $\sigma^{(n)}$ as in Proposition~\ref{pro:existence_sigma_and_sigma_n} \ref{item:L_n_sigma_n}.  
Below, we will show that 
\begin{equation} \label{eq:global_law} 
\lim_{n \to \infty} \P \bigg( \absbb{\frac{1}{n} \sum_{\xi \in \spec(X + A)} f(\xi) - \int_{\C} f \dd \sigma^{(n)}} >\eps \bigg) = 0.  
\end{equation}  
 for each $\eps>0$, $m \in \N$ and $f \in C_b(\C)$ with $\supp f \cap \Omega_m = \varnothing$.   
Here, $\Omega_m$ with $m \in \N$ is chosen as follows. Let $C>0$ be a constant such that $\abs{a(x) - a(y)} 
\leq C \abs{x-y}^\theta$ for all $x, y \in I_k$ and $k \in \db{K}$. The existence of such $C$ follows from 
\ref{assum:a_better_than_hoelder_1_2}. Then there are $x_1^{(m)},\ldots, x_{mK}^{(m)}$ such that 
$a([0,1]) \subset \bigcup_{i=0}^{mK} \DD_{Cm^{-\theta}}(x_i^{(m)})$ by \ref{assum:a_better_than_hoelder_1_2}. 
We set $\Omega_m := \bigcup_{i=0}^{mK} \DD_{2Cm^{-\theta}}(x_i^{(m)})$. Note that the Lebesgue measure $\abs{\Omega_m}$ of $\Omega_m$ tends to zero when $m \to \infty$ as $\theta >1/2$. 

We now justify that it suffices to prove \eqref{eq:global_law}. 
Then by Corollary~\ref{cor:convergence_sigma_n_to_sigma}, the convergence in 
\eqref{eq:global_law} holds when $\sigma^{(n)}$ is replaced by $\sigma$ from Proposition~\ref{pro:existence_sigma_and_sigma_n} \ref{item:L_sigma}. 
Fix $\eps>0$ and an arbitrary $f \in C_b(\C)$. For each $m \in \N$, we choose $\psi_m \in C_b(\C)$ such 
that $\ran \psi_m \subset [0,1]$, $\psi_m \equiv 0$ on $\Omega_{2m}$ and $\psi_m \equiv 1$ 
on $\C \setminus \Omega_{m}$. 
In particular, $f \psi_m \in C_b(\C)$ and $\supp(f\psi_m) \cap \Omega_{2m} = \varnothing$. 
Hence, $\frac{1}{n} \sum_{\xi} f(\xi) \psi_m(\xi)$ converges to $\int f\psi_m \dd\sigma$ in probability as $n \to \infty$, where the sum is taken over $\xi \in \spec(X + A)$.  
Furthermore, 
\begin{align*}  
\absbb{\frac{1}{n} \sum_{\xi \in \spec(X + A)} f(\xi) (1- \psi_m(\xi))} 
\leq &  \norm{f}_\infty \absbb{\int_{\C} 1 - \psi_m \dd \sigma - \frac{1}{n} \sum_{\xi \in \spec(X + A)} \psi_m(\xi) + 
\int_{\C} \psi_m \dd \sigma } \\ 
\leq & \norm{f}_\infty \bigg( \norm{\sigma}_\infty \abs{\Omega_{m}} + \absbb{\frac{1}{n} \sum_{\xi \in \spec(X + A)} \psi_m(\xi) - 
\int_{\C} \psi_m \dd \sigma } \bigg) 
\end{align*}  
where,  by a slight abuse of notation, we denoted by $\sigma$ the density of the measure $\sigma$ (cf.\ Proposition~\ref{pro:limiting_measure_and_S}) in the second step. 
Moreover, we used $\supp (1- \psi_m) \subset \Omega_{m}$ and the boundedness of the density $\sigma$.  
As argued above, the Lebesgue measure $\abs{\Omega_{m}}$ tends to zero as $m \to \infty$. 
Therefore, $\frac{1}{n} \sum_{\xi } f(\xi) (1 - \psi_m(\xi))$ converges to zero 
in probability when $n \to \infty$ due to Corollary~\ref{cor:convergence_sigma_n_to_sigma}, \eqref{eq:global_law} and $\supp \psi_m \cap \Omega_{2m} = \varnothing$. 
Using $\abs{\Omega_{m}} \to 0$ as $m \to \infty$ again, we see that $\int_{\C} f(\xi) (1 - \psi_m(\xi)) \sigma(\xi) \dd^2 \xi$ tends to zero as $m \to \infty$. 
This completes the proof of Theorem~\ref{thm:global_law} assuming that \eqref{eq:global_law} holds.

The main part of the proof  of \eqref{eq:global_law}   is to show the existence of a constant $\delta >0$ such that 
\begin{equation} \label{eq:global_law_X_quantitative} 
\absbb{\frac{1}{n} \sum_{\zeta \in \spec  (X + A)} f(\zeta) - \int_{\C} f(\zeta ) \sigma^{(n)}(\dd \zeta)} 
 \lesssim n^{-\delta} \norm{\Delta f}_{\mathrm{L}^3}
\end{equation}
with  probability at least $1-O(n^{-\delta})$  uniformly for all $f \in C_0^2 (\C)$ satisfying $\supp f \subset  \mathbb{D}_\varphi  \setminus \Omega_m$
for any fixed  constants $\varphi \in (0,\infty)$ and $m \in \N$.  

We now explain how \eqref{eq:global_law_X_quantitative} implies \eqref{eq:global_law}, and thus Theorem~\ref{thm:global_law}. 
If $f \in C_0^2 (\C)$  with $\supp f \cap \Omega_m = \varnothing$  then this is obvious. 
Let $f \in C_b(\C) \setminus C_0^2(\C)$  such that $\supp f \cap \Omega_m = \varnothing$.  
Owing to Remark~\ref{rem:exclusion_eigenvalues}, we know that $\spec(X + A) \subset \mathbb{S}_1$ with very high probability.  
We note that $\mathbb{S}_1 \subset \mathbb{D}_\varphi$ for some $\varphi \sim 1$ by Remark~\ref{rmk:S eps}. 
By possibly increasing $\varphi \sim 1$, we also have 
 $\supp \sigma^{(n)} \subset \mathbb{D}_{\varphi}$ 
due to Proposition~\ref{pro:existence_sigma_and_sigma_n} \ref{item:L_n_sigma_n}.  
Therefore, it suffices to consider $f \in C_b(\C)$ with $\supp f \subset \mathbb{D}_{\varphi + 1} \setminus \Omega_m$. 
Then we find $f_\eps \in C_0^2(\C)$ such that $\norm{f - f_\eps}_{\mathrm L^\infty}\leq \eps/2$, $\supp f_\eps \subset \mathbb{D}_{\varphi + 1} \setminus \Omega_m$ 
and $\norm{\Delta f_\eps}_{\mathrm L^3} \lesssim_\eps 1$. 
Hence, approximating $f$ by $f_\eps$ in \eqref{eq:global_law} and using \eqref{eq:global_law_X_quantitative} for $f_\eps$ 
shows that \eqref{eq:global_law_X_quantitative} implies \eqref{eq:global_law}. 

It remains to show \eqref{eq:global_law_X_quantitative}. 
We fix  constants $\varphi  \in (0,\infty)$ and $m \in \N$ and set $\Omega = \mathbb{D}_\varphi\setminus \Omega_m$.  
For any $T>0$, we conclude from \eqref{eq:spectral_statistics_log_determinant}, \eqref{eq:log_determinant_integral}, 
Proposition~\ref{pro:existence_sigma_and_sigma_n} \ref{item:L_n_sigma_n}  and  the second bound in \eqref{eq:integral_bounds} for $v_1^{(n)}$ that 
\begin{equation} \label{eq:linear_statistics_reformulation_proof_global_law}  
\frac{1}{n} \sum_{\xi \in \spec (X + A)} f(\xi) - \int_{\C} f(\zeta) \sigma^{(n)}(\dd \zeta) 
= \int_{\Omega} F(\zeta) \frac{\dd^2 \zeta}{\abs{\Omega}} + O\big( T^{-1} \norm{\Delta f}_{\mathrm{L}^1}), 
\end{equation}  
where 
\[ F(\zeta) := \frac{\abs{\Omega}}{2\pi} (\Delta f(\zeta)) h(\zeta), \qquad h(\zeta) := \frac{1}{n} \sum_{\xi 
\in \spec(X + A)} \log\abs{\xi- \zeta} + \int_0^T \bigg( \avg{v_1^{(n)}(\zeta,\eta)} - \frac{1}{1 + \eta} \bigg) \dd \eta. \] 
Note that $h$ and, thus, $F$ depend on the choice of $T$.

Before estimating $\int_{\Omega} F(\zeta) \frac{\dd^2\zeta}{\abs{\Omega}}$, we now prove a pointwise bound of $F$, 
which will be translated to a bound on $\int_{\Omega} F(\zeta) \frac{\dd^2\zeta}{\abs{\Omega}}$ later with the help 
of Lemma~\ref{lem:monte_carlo}. 
In fact, we now show that there are constants $\delta>0$ and $\alpha >0$ such that with $T := n^{\delta}$  
\begin{equation} \label{eq:F_small_pointwise} 
\abs{F(\zeta)}  \lesssim  n^{-\delta} \abs{\Delta f(\zeta)} 
\end{equation} 
 with probability at least $1 - O(n^{-\alpha})$  uniformly for all $\zeta \in \Omega$. 

 Fix $\zeta \in\Omega$. We choose $\delta>0$ such that $3\delta$ coincides with $\delta>0$ from Proposition~\ref{pro:global_law_H}.  
We set $\eta_* := n^{- 3 \delta}$ and introduce 
\begin{align*} 
h_1(\zeta) & := \int_{\eta_*}^T \big( \avg{v_1^{(n)}(\zeta, \eta)}  - \Im \avg{G(\zeta, \ii \eta)} \big) \dd \eta, & 
h_2(\zeta) & := - \int_0^{\eta_*} \Im \avg{G(\zeta, \ii \eta)} \dd \eta \\ 
h_3(\zeta) & := \frac{1}{4n} \sum_{\lambda \in \spec(H_\zeta)} \log \bigg( 1 + \frac{\lambda^2}{T^2} \bigg) 
 - \log \bigg( 1 + \frac{1}{T} \bigg), & 
h_4(\zeta) & := \phantom{-} \int_0^{\eta_*} \avg{v_1^{(n)}(\zeta, \eta)} \dd \eta.
\end{align*} 
Hence, owing to \eqref{eq:log_determinant}, \eqref{eq:log_determinant_integral} and $\int_0^T (1 + \eta)^{-1} \dd \eta = \log (1 + T)$, 
we obtain the decomposition $h(\zeta) = h_1(\zeta) + h_2(\zeta) + h_3(\zeta) + h_4(\zeta)$.

Next, we estimate the terms $h_1, \ldots, h_4$ individually. 
For $h_1$,  we note that $S^{(n)} = (\E \abs{x_{ij}}^2)_{i,j \in \db{n}}$ by \eqref{eq:S_n_a_n_discretised} 
and the assumptions of Theorem~\ref{thm:global_law}. Hence,  Proposition~\ref{pro:global_law_H}, a union bound and a continuity argument in $\eta$ imply $\abs{h_1(\zeta)} \leq n^{-1 +  3  P\delta}$ with very high probability. 
A simple computation shows that 
\[ -h_2(\zeta) = \frac{1}{4n} \sum_{\lambda \in \spec(H_\zeta)} \log \bigg ( 1 + \frac{\eta_*^2}{\lambda^2} \bigg) \leq \frac{1}{4n} \sum_{\lambda \in \spec(H_\zeta)\cap[-\eta_*^{1/2}, \eta_*^{1/2}]} \log \bigg(1 + \frac{\eta_*^2}{\lambda^2}\bigg) + \eta_*,  \] 
where in the last step we used that $\log ( 1 + \eta_*^2 \lambda^{-2}) \leq \log (1 +  \eta_*) \leq \eta_*$ 
if $\abs{\lambda} > \eta_*^{1/2}$. 
To estimate the remaining sum, we  will use \ref{assum:smallest_singular_value}.  
As $a([0,1]) \subset \bigcup_{i=0}^{mK} \DD_{Cm^{-\theta}}(x_i^{(m)})$ (see the beginning of this proof), 
 $\min\{ \abs{\zeta- a(i/n)} \colon i \in \db{n} \} \geq Cm^{-\theta}$ for all sufficiently large $n \in \N$. 
Therefore, owing to \ref{assum:smallest_singular_value} and Lemma~\ref{lem:small_singular_values}, 
we find a constant $\alpha >0$ such that  
\[ 
\frac{1}{4n} \sum_{\lambda \in \spec(H_\zeta)\cap[-\eta_*^{1/2}, \eta_*^{1/2}]} 
\log \bigg(1 + \frac{\eta_*^2}{\lambda^2}\bigg) 
\lesssim \frac{\log \eta_* + \abs{\log \min_{\lambda \in \spec(H_\zeta)} \abs{\lambda}}}{n} \abs{\spec(H_\zeta)\cap[-\eta_*^{1/2}, \eta_*^{1/2}]} \lesssim n^{\eps} \eta_*^{1/2} 
\] 
with  probability at least $1-O(n^{-\alpha})$  for any $\eps>0$. 
Therefore, $\abs{h_2(\zeta)} \lesssim n^{-\delta}$.
To estimate $h_3$, we use $\log (1 + x) \leq x$ and obtain 
\[ \abs{h_3(\zeta)} \leq \frac{1}{4nT^2} \tr(H_\zeta)^2 + T^{-1} = \frac{1}{2n T^2} \sum_{i,j=1}^n 
(\overline{x_{ji}} + (\bar a_i -  \bar \zeta)  \delta_{ji})(x_{ij} + ( a_i - \zeta) \delta_{ij}) + T^{-1} 
\lesssim T^{-1} \] 
since $\abs{x_{ij}} \leq n^{-1/2 + \eps}$ with very high probability due to \eqref{bounded moments} and $\abs{a_i} + \abs{\zeta} \lesssim 1$  as $\norm{a}_\infty\lesssim 1$  and $\zeta \in \mathbb{D}_\varphi$. 
  Since $s$ satisfies \ref{assum:primitive_upper_lower} and $a \in \mathcal B$, 
\eqref{eq:S_n_a_n_discretised} and 
 \eqref{eq:avg_M_bounded} for $\avg{v_1^{(n)}}$ imply  $\abs{h_4(\zeta)} \lesssim \eta_*$ uniformly for all $n \in \N$. 
This completes the proof of \eqref{eq:F_small_pointwise}.

Next, we use \eqref{eq:F_small_pointwise} and Lemma~\ref{lem:monte_carlo} to estimate $\int_{\Omega} F(\zeta) \frac{\dd^2\zeta}{\abs{\Omega}}$. 
Since $\zeta \mapsto \log \abs{\xi - \zeta}$ lies in $\mathrm{L}^p(\Omega)$ for every $p \in [1,\infty)$, 
the first bound in \eqref{eq:integral_bounds} implies that, for every $p \in [1,\infty)$, 
$\norm{h}_{\mathrm L^p(\Omega)} \lesssim_p 1$ uniformly for $T >0$. 
Therefore, $F \in L^2(\Omega)$ and Lemma~\ref{lem:monte_carlo} with $\eps = n^{-\alpha/4}$ and $N = n^{3\alpha/4}$ yields 
\begin{equation} \label{eq:integral_minus_difference_bound} 
\absbb{ \int_{\Omega} F(\zeta) \frac{\dd^2\zeta}{\abs{\Omega}} - \frac{1}{N} \sum_{i = 1}^N F(\xi_i) } 
\lesssim n^{-\alpha/4} \norm{F}_{\mathrm L^2} \lesssim n^{-\alpha/4} \norm{\Delta f}_{\mathrm L^3} 
\end{equation} 
with probability at least $1-n^{-\alpha/4}$, where $\xi_1, \ldots, \xi_N$ are independent random variables distributed according to 
the normalized Lebesgue measure on $\Omega$  and independent of $X$ for all $n \in \N$.

Furthermore, conditioning on $\xi_1$, \ldots, $\xi_N$, a union bound over $i \in \db{N}$ and the bound \eqref{eq:F_small_pointwise} imply 
\begin{equation}\label{eq:F_average_small} 
\frac{1}{N} \sum_{i=1}^N \abs{F(\xi_i)} \leq \frac{n^{-\delta}}{N}\sum_{i=1}^m \abs{\Delta f(\xi_i)} 
\leq n^{-\delta} \norm{\Delta f}_{\mathrm L^1} + n^{- \alpha/4} \norm{\Delta f}_{\mathrm L^2} 
\end{equation}  
with  probability at least $1 - O(n^{-\alpha/4})$,  where the second step follows from Lemma~\ref{lem:monte_carlo} with $F = \Delta f$ as well 
as $\eps = n^{-\alpha/4}$ and $N=n^{3\alpha/4}$ as before. 
Finally, we combine \eqref{eq:linear_statistics_reformulation_proof_global_law}, \eqref{eq:integral_minus_difference_bound} and \eqref{eq:F_average_small},  recall the choice $T = n^{\delta}$ 
and choose $\delta$ to be $\min\{ \delta, \alpha/4\}$ 
to obtain \eqref{eq:global_law_X_quantitative}. This  completes the proof of Theorem~\ref{thm:global_law}. 
\end{proof}

\linkdest{proof:thr:Spectrum occupies pseudospectrum}
\subsection{Proof of Theorem~\ref{thr:Spectrum occupies pseudospectrum}} 
\label{subsec:proof_theorem_spectrum_pseudospectrum} 

We recall the definition of $\mathbb S_\eps$ from \eqref{eq:def_S_eps} and set 
\begin{equation} \label{eq:def_spec_eps_infty_s_a} 
\spec_\eps^\infty(s,a) := \mathbb S_\eps.  
\end{equation} 
With this definition, \eqref{eq:cap_spec_eps_infty_equals_supp_sigma} follows from 
$\cap_{\eps>0} \spec_\eps^\infty(s,a) = \mathbb S_0 = \supp \sigma$ due to 
Remark~\ref{rmk:S eps} and \eqref{eq:overline_S_equal_intersection_S_eps}. 
Now we verify \eqref{convergence of pseudospectrum}. First we see that for any $\eps,\delta>0$ the inclusion 
\[
\limsup_{n\to \infty} \spec_\eps(X_n + A_n) \subset \spec_{\eps+\delta}^\infty(s, a)= \mathbb S_{\eps+\delta}
\]
holds almost surely by Remark~\ref{rem:exclusion_eigenvalues} and the Borel-Cantelli lemma. Since $\cap_{\delta>0}\mathbb S_{\eps+\delta}=\mathbb S_{\eps}$ by definition this shows the first inclusion in \eqref{convergence of pseudospectrum}.

The second inclusion  in \eqref{convergence of pseudospectrum}  follows from  
\bels{second inclusion of thr 2.3}{
\mathbb S_{\eps} \subset \spec_{\eps+\delta }(X_n + A_n)= \{\zeta \in \C: \dist(0,\spec(H_\zeta)) \le \eps +\delta \}
}
eventually almost surely for any $\eps,\delta>0 $.   Here $H_\zeta$ is the Hermitisation of $X_n+A_n$ from \eqref{eq:def_H_zeta}. 
To prove \eqref{second inclusion of thr 2.3} we see that the global law from \cite[Theorem~2.7]{AEKN_Kronecker} holds almost surely when all random matrices in the statement are realised on the same probability space. This can be seen easily from its proof. Indeed, the global law is an immediate consequence of  \cite[eq.\ (B.5)]{AEKN_Kronecker}, which holds with very high probability. Thus, the Borel-Cantelli lemma  and Corollary~\ref{cor:singular_value_density_n_to_rho_zeta} \ref{item:convergence_rho_zeta_n} ensure  almost sure convergence in  
\[
\frac{1}{2n}\tr f(H_\zeta) \to \int_{\R} f(\tau) \rho_{\zeta}(\dd \tau)
\]
for every compactly supported continuous function $f$.
\qed

\section{Discretizing the Dyson equation} 

\label{sec:discretizing_Dyson_equation} 

 In this section, we prove Corollary~\ref{cor:convergence_sigma_n_to_sigma} and Corollary~\ref{cor:singular_value_density_n_to_rho_zeta}. They  both follow from the next lemma.

Throughout this section, we write $\C_+ := \{ w \in \C \colon \Im w >0 \}$.

\begin{lemma} \label{lem:discretizing_dyson_eq} 
Let $s$ and $a$ satisfy  \ref{assum:s_a_piecewise_continuous}.  
Let $M(\zeta,w)$ be the solution of \eqref{general MDE} associated with $s$ and $a$. 
For $n \in \N$, define $\wh{a}^{(n)}\colon [0,1] \to \C$ and $\wh{s}^{(n)}\colon [0,1]^2 \to [0,\infty)$ through 
\begin{equation} \label{eq:def_a_n_s_n}  
\wh{a}^{(n)} := \sum_{i=1}^{n} a(i/n) \mathbf 1_{[(i-1)/n, i/n)}, \qquad 
\wh{s}^{(n)} := \frac{1}{n}\sum_{i,j = 1}^n s(i/n,j/n) \mathbf 1_{[(i-1)/n, i/n)\times [(j-1)/n, j/n)}, 
\end{equation}  
where $\mathbf 1_\Omega$ denotes the indicator function of the set $\Omega$.  
Let $\wh{\Sigma}^{(n)}$ be defined analogously to \eqref{eq:def_Sigma} with $s$ replaced by $\wh{s}^{(n)}$.  
If $\wh{M}^{(n)}$ is the unique solution of \eqref{general MDE} with $\wh{a}^{(n)}$ and $\wh{\Sigma}^{(n)}$ instead of $a$ and $\Sigma$, 
$\delta >0$ is constant and $\zeta \in \C$ is fixed, then  
\[ \lim_{n \to \infty} \norm{\wh{M}^{(n)}(\zeta, w) - M(\zeta, w)}_2 = 0 \] 
uniformly for all $w \in \C_+$ satisfying $\dist(w,\supp \rho_\zeta) \geq \delta$. 
Here, $\norm{R}_2 := \norm{\Tr(R^*R)}_{1}^{1/2}/\sqrt{2}$ for any $R \in \mathcal B^{2\times 2}$, where $\Tr(R^*R)$ is considered as a function on $[0,1]$, 
and $\norm{f}_p$ is the $L^p([0,1], \mu)$-norm for $f \colon [0,1] \to \C$. 
\end{lemma}

Throughout the remainder of this section, some operators appear that map $\mathcal B^{2\times 2}$ 
to $\mathcal B^{2\times 2}$. We write $\norm{\,\cdot\,}_{\ast \to \#}$ with $\ast$, $\# \in \{ 2, \infty\}$ 
for the operator norm if the definition space is equipped 
with the norm $\norm{\,\cdot\,}_\ast$ and the target space with $\norm{\,\cdot\,}_\#$. 
If $\ast = \#$ then we simply write $\norm{\,\cdot\,}_\ast$ for the corresponding operator norm.

\begin{proof} 
We fix $\delta >0$ and $\zeta \in \C$. 
We introduce the matrices $A\in \mathcal B^{2\times 2}$ and $A^{(n)}\in \mathcal B^{2\times 2}$ through  
\[ 
A := \begin{pmatrix} 0 & a\\ \overline{a} & 0 \end{pmatrix}, \qquad 
A^{(n)} := \begin{pmatrix} 0 & a^{(n)}\\ \overline{a^{(n)}} & 0 \end{pmatrix}. 
\] 
For $w \in \C_+$ satisfying $\dist(w, \supp \rho_\zeta )\geq \delta$ and $t \geq 0$, we set $\wh{M}^{(n)} = \wh{M}^{(n)}(\zeta,w + \ii t)$,  $M = M(\zeta,w+ \ii t)$ 
and 
$L[R] := R - M \Sigma[R] M$ for all $R \in \mathcal B^{2\times 2}$. 
With $\Delta := \wh{M}^{(n)} - M$, a short computation starting from \eqref{general MDE} and the 
analogous relation with $M^{(n)}$, $a^{(n)}$ and $\Sigma^{(n)}$ yields 
\bels{discretisation stability}{
L [\Delta] = M \Sigma[\Delta]\Delta + 
  M (\Sigma^{(n)} - \Sigma)[\wh{M}^{(n)}] \wh{M}^{(n)} + M(A - A^{(n)} ) \wh{M}^{(n)}. 
 }
We now invert $L$ and estimate the resulting relation in $\norm{\,\cdot\,}_2$. 
We collect a few auxiliary bounds. 
From \cite[eq.s~(3.22), (3.11a), (3.11c)]{AEKN_Kronecker}, we conclude the existence of a constant $C_1>0$, depending only on $\delta$ 
but independent of $w$ and $t$, 
such that $\norm{L^{-1}}_{2} \leq C_1$ for all $w \in \C_+$ with $\dist(w,\supp\rho_\zeta) \geq \delta$ and $t \geq 0$.  
 As $\norm{M(\zeta,w)} \leq (\dist(w,\supp \rho_\zeta))^{-1}$ by \cite[eq.~(3.11a)]{AEKN_Kronecker}\footnote{The proof in \cite{AEKN_Kronecker} is given in the finite dimensional setup; the proof in the setup of this article 
is identical.}, 
we have  
 $\norm{M} \leq  (\max\{\delta,t\})^{-1}$ and 
 $\norm{\wh{M}^{(n)}} \leq  t^{-1}$  for all $t \geq 0$. 
Owing to \cite[Lemma~B.2(i)]{AEK_Shape}, the upper bound on $s$ following from its piecewise continuity implies that 
there is a constant $C_2 \geq 1$ such that $\norm{\Sigma}_{2\to \infty} \leq C_2$. 
 From \ref{assum:s_a_piecewise_continuous} and \eqref{eq:def_a_n_s_n}, we conclude that  
\cite[Assumption~A2]{AK_Brown}   
holds with $\mathfrak X = [0,1]$, $\mu$ the Lebesgue measure on $[0,1]$ and 
$a$ and $s$ replaced by $\wh{a}^{(n)}$ and $\wh{s}^{(n)}$, respectively. 
Hence,  \cite[Lemma~5.5]{AK_Brown}  implies that for some constant $C_2>0$ 
depending only on $\delta$, we have $\norm{\wh{M}^{(n)}} \leq C_2$ if $\norm{\Delta}_2 = \norm{\wh{M}^{(n)} - M}_2 \leq 1$.  
Therefore, there is a constant $C\geq 1$, depending only on $ \delta$ but not on  $w$ or $t$,  such that 
\begin{equation} \label{eq:Delta_leq_Psi_plus_Delta_squared}  
\norm{\Delta}_2 \leq C( \norm{\Delta}_2^2 + \Psi_n), \qquad \Psi_n :=  \norm{\Sigma^{(n)} - \Sigma}_{2 } 
 + \norm{A- A^{(n)}}_2  
\end{equation} 
for all $w \in \C_+$  and all $t\geq 0$ satisfying 
 $\dist(w,\supp \rho_\zeta) \geq \delta$ and $\norm{\Delta}_2 \leq 1$.  Here, $\Delta \equiv \Delta (\zeta, w + \ii t)$. 

 Since $s$ and $a$ are blockwise uniformly continuous by \ref{assum:s_a_piecewise_continuous},  
 $\Psi_n \to 0$ as $n \to \infty$. 
Thus, we find $n_0 \in \N$ such that $2 \Psi_n C^2 \leq 1/4$ for all $n \geq n_0$. 
Fix $w \in \C_+$ with $\dist(w,\supp\rho_\zeta) \geq \delta$. 
We set $t_* := \sup\{ t \geq 0 \colon \norm{\Delta(\zeta,w + \ii t)}_2 \geq 2 C \Psi_n \}$. 
Since $\norm{M^{(n)}} + \norm{M} \to 0$ for $ t \to \infty$, we obtain $t_* < \infty$. 
Next, we conclude $t_* = 0$. Suppose $t_* >0$. 
Hence, $\norm{\Delta(\zeta, w + \ii t_*)}_2 = 2 C \Psi_n$ by continuity. 
As $2 \Psi_n C^2 \leq 1/4$, we  deduce $\norm{\Delta(\zeta, w + \ii t_*)}_2 \leq 1$ and hence,  from \eqref{eq:Delta_leq_Psi_plus_Delta_squared} that $\norm{\Delta(\zeta, w + \ii t_*)}_2 \leq 3 C\Psi_n/2 < 2 C \Psi_n = \norm{\Delta(\zeta, w + \ii t_*)}_2$. This contradiction implies $t_*=0$. 
Note that this holds for any $w \in \C_+$ as long as $\dist(w,\supp\rho_\zeta) \geq \delta$ and $n \geq n_0$. 
Thus, for $n \geq n_0$, we obtain 
$ \norm{M^{(n)}(\zeta, w) - M(\zeta,w)}_2 = \norm{\Delta(\zeta, w)}_2 \leq 2 C \Psi_n $ 
for all $w \in \C_+$ with $\dist(w,\supp\rho_\zeta) \geq \delta$, 
which concludes the proof of Lemma~\ref{lem:discretizing_dyson_eq} as $\Psi_n \to 0$ with $n\to \infty$. 
\end{proof}

Before proving Corollary~\ref{cor:convergence_sigma_n_to_sigma}, we remark that if  $a \in \mathcal B$ and  $s$ satisfies \ref{assum:primitive_upper_lower} 
 then 
\begin{equation} \label{eq:v_1_large_eta} 
 \norm{v_1(\zeta, \eta) - (1 + \eta)^{-1}} \lesssim (1 + \abs{\zeta})\eta^{-2} 
\end{equation} 
uniformly for $\eta \geq 1$ and $\zeta \in \C$  due to \cite[eq.~(5.15)]{AK_Brown}.

\begin{proof}[\linkdest{proof:cor:convergence_sigma_n_to_sigma}{Proof of Corollary~\ref{cor:convergence_sigma_n_to_sigma}}]   
Since $\sigma^{(n)}$ for all $n \in \N$ and $\sigma$ are probability measures on $\C$, for the weak convergence 
it suffices to show $\int_{\C} f \dd \sigma^{(n)} \to \int_{\C} f \dd \sigma$ as $n \to \infty$ 
for all $f \in C_0^2(\C)$. 
Fix $f \in C_0^2(\C)$. 
 As $a \in \mathcal B$ and $s$ satisfies \ref{assum:primitive_upper_lower},  
we conclude from  \eqref{eq:avg_M_bounded},  \eqref{eq:v_1_large_eta}  and the compactness of $\supp f$  that 
$ \absb{\Delta f(\zeta)\big( \avg{v_1^{(n)}(\zeta, \eta)} - \frac{1}{1 + \eta} \big)} 
\lesssim \frac{\abs{\Delta f(\zeta)}}{1 + \eta^2}$ 
uniformly  for  $\eta >0$, $\zeta \in \C$ and $n \in\N$. That is the implicit constant hidden by $\lesssim$ 
does not depend on $\eta$, $\zeta$ and $n$. 
Owing to the integrability of the right-hand side with respect to $\zeta$ and $\eta$ over $\C \times (0,\infty)$, 
we obtain from \eqref{eq:sigma_L_identity}, Fubini, dominated convergence  and Lemma~\ref{lem:discretizing_dyson_eq} with $\avg{M(\zeta,\ii\eta)} = \ii \avg{v_1(\zeta,\eta)}$ (compare \eqref{eq:structure_M}) and $\avg{\wh{M}^{(n)}(\zeta,\ii\eta)}=\avg{M^{(n)}(\zeta,\ii\eta)} = \ii \avg{v_1^{(n)}(\zeta,\eta)}$ for all $\zeta \in \C$ and $\eta >0$   that 
\begin{multline*}  
\int_{\C} f \dd \sigma^{(n)} = \int_{\C} \int_0^\infty \Delta f(\zeta) \bigg( \avg{v_1^{(n)}(\zeta, \eta)} 
- \frac{1}{1+ \eta} \bigg) \dd \eta \,  \dd^2 \zeta \\ 
 \longrightarrow \int_{\C} \int_0^\infty \Delta f(\zeta) \bigg( \avg{v_1(\zeta, \eta)} 
- \frac{1}{1+ \eta} \bigg) \dd \eta \,  \dd^2 \zeta = \int_{\C} f \dd \sigma
\end{multline*}  
as $n \to \infty$. This completes the proof of Corollary~\ref{cor:convergence_sigma_n_to_sigma}. 
\end{proof}

\begin{proof}[\linkdest{proof:cor:singular_value_density_n_to_rho_zeta}{Proof of Corollary~\ref{cor:singular_value_density_n_to_rho_zeta}}]
Item \ref{item:convergence_rho_zeta_n} follows directly from the convergence of the Stieltjes transforms, 
i.e.\ for each $w \in \C_+$, $\frac{1}{2n} \tr M^{(n)}(\zeta,w)  = \avg{\wh{M}^{(n)}(\zeta,w)}  \to \avg{M(\zeta,w)}$ as $n \to \infty$, due to Lemma~\ref{lem:discretizing_dyson_eq}. 

For the proof of \ref{item:limsup_rho_zeta_n_subset_supp_rho_zeta}, 
 it suffices to show for fixed $\delta>0$ that $\supp \rho_\zeta^{(n)} \subset \supp \rho_\zeta+(-\delta, \delta)$ for all sufficiently large $n$. Fix $\delta >0$.  
If $\tau \in \R$ satisfies $\dist(\tau, \supp \rho_\zeta) \ge \delta $ then, by \cite[Lemma D.1]{AEK_Shape}, 
$M=M(\zeta, \tau) = \lim_{\eta \downarrow 0}M(\zeta, \tau + \ii \eta) $ exists and is self-adjoint. Moreover, $\norm{L^{-1}}_{2 }+ \norm{M}\lesssim_\delta 1$ uniformly for $\eta \geq 0$ and $\tau \in \R$ with $\dist(\tau,\supp \rho_\zeta) \geq \delta$. 
We recall the definition $L[R]= R- M\Sigma[R]M$ for $R \in \mathcal B^{2\times 2}$ from the proof of Lemma~\ref{lem:discretizing_dyson_eq}. 
As $\norm{M} \lesssim_\delta 1$, Lemma~\ref{lem:discretizing_dyson_eq} implies $\norm{\wh{M}^{(n)}(\zeta,\tau + \ii \eta)}_{2} \lesssim_\delta 1$ uniformly for all $\eta >0$, $\tau\in \R$ with $\dist(\tau,\supp\rho_\zeta) \geq \delta$ 
and all sufficiently large $n$. 
Arguing similarly as in the proof of  Lemma~\ref{lem:discretizing_dyson_eq},  we conclude 
$\norm{\wh{M}^{(n)}(\zeta,\tau + \ii \eta)} \lesssim_\delta 1$ uniformly for $\eta$, $\tau$ and $n$ as before. 
We set $\wh{M}^{(n)} := \wh{M}^{(n)}(\zeta, \tau + \ii \eta)$ and $L^{(n)}[R] = R- \wh{M}^{(n)}\Sigma^{(n)}[R] \wh{M}^{(n)}$ for $R \in \mathcal B^{2\times 2}$. 
For such $\eta$, $\tau$ and $n$, we obtain $\norm{(L^{(n)})^{-1}}_{2} \lesssim_\delta 1$ by perturbation 
theory from $\norm{M} + \norm{\wh{M}^{(n)}} + \norm{L^{-1}}_{2} \lesssim_\delta 1$, 
$\norm{\Sigma^{(n)}}_{2 \to \infty} + \norm{\Sigma}_{2\to \infty} \lesssim 1$ 
and $\norm{\wh{M}^{(n)} - M}_{2} \to 0$ for $n \to \infty$. 
Hence, by the implicit function theorem, for all sufficiently large $n$, the function $\eta \mapsto 
\wh{M}^{(n)}(\zeta, \tau + \ii \eta)$ is continuous on $[\eta_0-\eps,\eta_0 + \eps]$ for some $\eps>0$ independent of $\eta_0>0$. 
In particular, we can extend $\wh{M}^{(n)}$ continuously to $\eta = 0$ in a unique way.

Let $\tau \in \R$ with $\dist(\tau,\supp\rho_\zeta) \geq \delta$. 
For $M = M(\zeta,\tau)$, we now consider the relation 
\[
L[\Delta] = \frac{1}{2}\big(K_n(\Delta,\wt{\Sigma},\wt{A})+K_n(\Delta^*,\wt{\Sigma},\wt{A})^*\big)\,, \quad 
K_n(\Delta):= M \Sigma[\Delta]\Delta + 
  M \wt{\Sigma}[M+\Delta] (M+\Delta) + M\wt{A}(M + \Delta)\,,
\]
with variables $\Delta \in \mathcal B^{2\times 2}$, $\wt{A} =\wt{A}^* \in \mathcal B^{2\times 2}$, $\wt{\Sigma} \colon \mathcal B^{2\times 2} \to \mathcal B^{2\times 2}$ such that $\wt{\Sigma}[R]^* = \wt{\Sigma}[R^*]$ for all $R \in \mathcal B^{2\times 2}$. 
Since $\norm{L^{-1}}_{2}\lesssim_\delta 1$, by the implicit function theorem, this relation has a unique 
solution $\Delta$ as long as $\norm{\wt{\Sigma}}_{2}$ and $\norm{\wt{A}}_2$ are sufficiently small, 
as $L[0] = 0$ and $K_n(0,0,0) = 0$. 
Moreover, this solution satisfies $\Delta = \Delta^*$ as $L[R]^* = L[R^*]$ for all $R \in \mathcal B^{2\times 2}$ due to $M^* = M$. 
Owing to \eqref{discretisation stability} and $M=M^*$, we have $L[\wh{M}^{(n)} - M] 
= (K_n(\wh{M}^{(n)} - M,\Sigma^{(n)}-\Sigma,A-A^{(n)}) + K_n((\wh{M}^{(n)} - M)^*,\Sigma^{(n)}-\Sigma,A-A^{(n)}))/2$ 
with $\wh{M}^{(n)} = \wh{M}^{(n)}(\zeta,\tau)$.
Hence, as $\norm{\Sigma^{(n)}-\Sigma}_{2} + \norm{A- A^{(n)}}_2 \to 0$ for $n \to \infty$ by the proof of Lemma~\ref{lem:discretizing_dyson_eq}, we get  $\Delta = \wh{M}^{(n)}- M$ for all sufficiently large $n$ and, therefore, 
$\wh{M}^{(n)} = (\wh{M}^{(n)})^*$ for such $n$.   
Since this holds for any $\delta >0$ we conclude that $\im \wh{M}^{(n)}(\zeta, \tau+\omega)=0$ for sufficiently small $\abs{\omega}$ with $\omega \in \R$. 
Because of  $\frac{1}{2n} \tr M^{(n)}(\zeta,w) = \avg{\wh{M}^{(n)}(\zeta,w)}$ and  \eqref{ST of rho n}, this implies that $\tau \not \in \supp \rho_\zeta^{(n)}$. 
\end{proof}

{\small 
\bibliography{bibliography} 
\bibliographystyle{amsplain-nodash} 
}

\end{document}